\documentclass[a4paper, oneside]{article}

\usepackage{amsmath}
\usepackage{amssymb}
\usepackage{amsthm}
\usepackage{appendix}
\usepackage{enumerate}
\usepackage{hyperref}
\usepackage{mathrsfs}
\usepackage{pdfsync}
\usepackage{titlesec}

\titleformat{\section}[hang]
{\bfseries\large}{\thesection.}{0.5em}{}[]
\titlespacing*{\section}{0em}{2em}{1.5em}

\titleformat{\subsection}[runin]
{\bfseries\normalsize}{\thesubsection.}{0em}{\ }[.]

\theoremstyle{plain}
\newtheorem{thm}{Theorem}[section]
\newtheorem{prop}[thm]{Proposition}
\newtheorem{lem}[thm]{Lemma}
\newtheorem{cor}[thm]{Corollary}
\newtheorem*{conj}{Conjecture}

\theoremstyle{definition}
\newtheorem{defn}[thm]{Definition}

\theoremstyle{remark}
\newtheorem{rem}[thm]{Remark}

\numberwithin{equation}{section}

\DeclareMathOperator*{\essinf}{essinf}
\DeclareMathOperator*{\esslim}{esslim}

\DeclareMathOperator{\sgn}{sgn}

\newcommand{\Romannum}[1]{\uppercase\expandafter{\romannumeral#1\relax}}
\newcommand{\romannum}[1]{\romannumeral#1\relax}

\newcommand \ep {\epsilon}
\newcommand \ve {\varepsilon}

\newcommand \vf {\varphi}

\newcommand \bE {\mathbb E}
\newcommand \bN {\mathbb N_+}
\newcommand \bP {\mathbb P}
\newcommand \bQ {\mathbb Q}
\newcommand \bR {\mathbb R}

\newcommand \bZ {\mathbb Z}

\newcommand \cA {\mathcal A}
\newcommand \cC {\mathcal C}
\newcommand \cS {\mathcal S}

\newcommand \fD {\mathfrak D}
\newcommand \fQ {\mathfrak Q}

\newcommand \CS {Cauchy--Schwarz }

\newcommand \Hol {H\"older's }

\def \ss {\mathrm{ss}}
\newcommand \ta {\mathrm{tas}}

\begin{document}

\pagestyle{plain}
\pagenumbering{arabic}
\bibliographystyle{plain}

\title{Hydrodynamics for one-dimensional ASEP\\in contact with a class of reservoirs}
\author{\textsc{Lu XU}\footnote{This work has been funded by the ANR grant MICMOV (ANR-19-CE40-0012) of the French National Research Agency (ANR).}}
\date{}
\maketitle



\begin{abstract}
We study the hydrodynamic behaviour of the asymmetric simple exclusion process on the lattice of size $n$.
In the bulk, the exclusion dynamics performs rightward flux.
At the boundaries, the dynamics is attached to reservoirs.
We investigate two types of reservoirs: (1) the reservoirs that are weakened by $n^\theta$ for some $\theta<0$ and (2) the reservoirs that create particles only at the right boundary and annihilate particles only at the left boundary.
We prove that the spatial density of particles, under the hyperbolic time scale, evolves with the entropy solution to a scalar conservation law on $[0,1]$ with boundary conditions.
The boundary conditions are characterised by the boundary traces \cite{BLN79,Otto96,Vass01} at $x=0$ and $x=1$ which take values from $\{0,1\}$.
\end{abstract}



\bigskip
\noindent\textbf{Data Availability.}
The authors declare that all data supporting this article are available within the article.

\section{Introduction}

The simple exclusion in contact with reservoirs is one of the most studied open interacting particle systems.
It can be viewed as a superposition of nearest-neighbour random walks on the lattice $\{1,\ldots,n\}$, in accordance with the exclusion rule: two particles cannot occupy the same site at the same time.
The dynamics is attached to reservoirs at its boundaries, which means that particles can enter the system at site $1$ with rate $n^\theta\alpha$ and leave with rate $n^\theta\gamma$, while at site $n$, similar behaviour is exhibited with rates $n^\theta\delta$ and $n^\theta\beta$.
Assume that $\alpha$, $\beta$, $\gamma$ and $\delta$ are $L^\infty$ functions of time.
The factor $n^\theta$ regulates the intensity of the contact between the bulk dynamics and the reservoirs.
The total number of particles turns out to be the only quantity that is conserved locally by the dynamics.

The purpose of this article is to investigate the density of particles as $n\to\infty$, namely the hydrodynamic limit.
More precisely, we want to prove that, under proper time scale, the empirical distribution of particles converges to the solution to some partial differential equation.
This equation governs the macroscopic time evolution of the local equilibrium states of the dynamics.

When the jump rate of the random walk is symmetric, boundary-driven heat equation is obtained in the limit under the diffusive time scale.
Depending on the value of $\theta$, the boundary conditions are determined through three phases: Dirichlet type boundaries for $\theta>-1$ (see \cite{ELS90,ELS91,KipnisL99,LMO08} and the references therein), Robin type boundaries for $\theta=-1$, and Neumann type boundaries for $\theta<-1$ (see \cite{BMNS17,FGN19}).

For asymmetric dynamics, the situation differs drastically.
The one-dimensional asymmetric simple exclusion process (ASEP) in contact with reservoirs is first introduced as an intermediate tool for studying the dynamics on the infinite lattice \cite{Liggett75}.
Its non-equilibrium stationary state (NESS) has uniform density determined by the bulk and reservoir rates through three phases: the high-density phase, the low-density phase, and the max-current phase \cite{DEHP93}.
The high-density phase and the low-density phase intersect at a critical line, where the density performs a randomly located shock.
The density solves the variational problem \cite{PopkovS99,Baha12,DMOX22}: the stationary flux is minimised when the density gradient created by the reservoirs has opposite sign to the drift of the asymmetric exclusion (drift up-hill), otherwise it is maximised (drift down-hill).

For ASEP, hydrodynamic limit is proved under the hyperbolic time scale.
The macroscopic time evolution is governed by a nonlinear scalar conservation law \cite{Reza91}.
As this equation does not in general have a unique weak solution, we are forced to consider the physical one obtained through parabolic perturbation, namely the \emph{entropy solution}.
When reservoirs are attached, the entropy solution is constrained by specific boundary conditions \cite{Baha12,Xu21}.
Different from the classical case, these boundary conditions do not fix the boundary values of the solution.
Instead, they impose a set of possible boundary values depending on the given boundary data, see \cite{BLN79,Otto96} and Proposition \ref{prop:ent-sol}.
Hence, the entropy solution exhibits \emph{boundary layers}, a typical discontinuous phenomenon happening in boundary-driven hyperbolic equations.
The characterisation of these layers turns out to be the main issue in proving hydrodynamic limit,  since it restricts us from employing replacement lemma for small macroscopic blocks.
When $\theta>0$, the reservoirs dominate the drift at the boundaries, thus the boundary data are set to be the reversible densities of the reservoirs, see \cite{Xu21} and Theorem \ref{thm:fast}.
The proof in \cite{Xu21} exploits a grading scheme to control the formulation of boundary layers on a mesoscopic level.
In \cite{DMOX22}, the time scales longer than the hyperbolic one are considered, under which the density converges to the quasi-stationary solution, that is, the stationary solution associated to dynamical boundary data.
Their result generalises the usual hydrostatic scenario, where only stationary dynamics is evolved.
When $\theta=0$, the hydrodynamic limit is proved in \cite{Baha12} for a special choice of $(\alpha,\beta,\gamma,\delta)$ which is constant in time and is dependent on the jump rates of the bulk random walk.
This choice approximates optimally the infinite dynamics \cite{Liggett75}, thus an effective coupling method is applicable in proving the boundary conditions, see Remark \ref{rem:liggett}.

In the present article, we focus on the regime $\theta\le0$.
We consider the ASEP with rightward flux and treat two types of boundary reservoirs:
\begin{itemize}
\item[(1)] $\theta<0$, i.e., the reservoirs are attached weakly to the bulk dynamics;
\item[(2)] $\theta=0$, the rates of enter at site $1$ and leave at site $n$ are $0$, i.e., particles are randomly destructed  on the left side and created on the right side.
\end{itemize}
Observe that the case (2) is not covered by the model studied in \cite{Baha12}.
In the main result, Theorem \ref{thm:slow}, we prove that the hydrodynamic equation is given by the scalar conservation law with boundary conditions formally read $u|_{x=0}=0$, $u|_{x=1}=1$.
Rigorously speaking, these conditions mean that the solution has $L^\infty$ boundary traces taking values from $\{0,1\}$, see Remark \ref{rem:slow}.
In both cases, the hydrodynamic equation conserves the total mass, see \eqref{eq:cl-amount}.
We also study in Theorem \ref{thm:viscous} the vanishing viscosity limit of the diffusive hydrodynamic equation for the corresponding weakly asymmetric simple exclusion for all values of $\theta$.
Based on this result, a conjecture on the hydrodynamic limit for $\theta=0$ with general reservoir rates is made in Section \ref{subsec:viscous}.

The proofs of the two types have almost no difference.
The bulk equation is proved in Section \ref{sec:local-equil} and \ref{sec:com-com} employing the logarithmic Sobolev inequality \cite{Yau97} and the compensated compactness argument \cite{Fritz04,FritzT04}, similarly to \cite{Xu21}.
The boundary conditions are formulated through a strategy different from the existing works \cite{Baha12,Xu21}.
We apply Vasseur's result on the existence of strong boundary traces \cite{Vass01} to reduce the boundary conditions to integral formulas.
These formulas are then verified in Section \ref{sec:ent-sol}, based on the fact that in both cases, only vanishing macroscopic current can be observed near the boundaries.

\section{Model and main results}

\subsection{ASEP with open boundaries}
\label{subsec:asep}

For $n\ge2$, let $\Omega_n=\{0,1\}^n$ be the configuration space.
For $\eta=(\eta_1,\ldots,\eta_n)\in\Omega_n$, define $\eta^{i,i+1}$ and $\eta^i $ by
\begin{align}
  (\eta^{i,i+1})_{i'}=
  \begin{cases}
    \eta_{i'+1}, &i'=i,\\
    \eta_{i'-1}, &i'=i+1,\\
    \eta_{i'}, &\text{otherwise},
  \end{cases}
  \quad
  (\eta^i)_{i'}=
  \begin{cases}
    1-\eta_{i'}, &i'=i,\\
    \eta_{i'}, &i'\not=i.
  \end{cases}
\end{align}
Let $L_\ta$ and $L_\ss$ generate respectively the totally asymmetric simple exclusion and symmetric simple exclusion on $\Omega_n$:
\begin{equation}
  \begin{aligned}
    L_\ta f(\eta) &:= \sum_{i=1}^{n-1} \eta_i(1-\eta_{i+1})\big[f(\eta^{i,i+1})-f(\eta)\big],\\
    L_\ss f(\eta) &:= \sum_{i=1}^{n-1} \big[f(\eta^{i,i+1})-f(\eta)\big].
  \end{aligned}
\end{equation}
Given nonnegative functions $\alpha$, $\beta$, $\gamma$, $\delta \in L^\infty(\bR_+)$, define
\begin{equation}
  \begin{aligned}
    L_{-,t}f(\eta) &:= \big[\alpha(t)(1-\eta_1)+\gamma(t)\eta_1\big]\big[f(\eta^1)-f(\eta)\big],\\
    L_{+,t}f(\eta) &:= \big[\delta(t)(1-\eta_n)+\beta(t)\eta_n\big]\big[f(\eta^n)-f(\eta)\big].
  \end{aligned}
\end{equation}
Let $\{\mu_{n,0}; n\ge2\}$ be a sequence of probability measures on $\Omega_n$.
Consider the Markov process $\eta=\{\eta(t);t\ge0\}$ starting from $\mu_{n,0}$ and generated by the operator
\begin{equation}
\label{eq:generator}
  \begin{aligned}
  L_{n,t} &:= pnL_\ta+\sigma n^{1+\kappa} L_\ss+n^{1+\theta}L_{-,t}+n^{1+\theta}L_{+,t}\\
  &= n\big[pL_\ta+\sigma n^\kappa L_\ss+n^\theta(L_{-,t}+L_{+,t})\big],
  \end{aligned}
\end{equation}
where $p>0$, $\sigma>0$, $\kappa\in(0,1)$ and $\theta\in\bR$ are constants.
In the last line of \eqref{eq:generator}, the factor $n$ corresponds to the hyperbolic time scale.
The symmetric exclusion is speeded up by $n^\kappa$, $\kappa>0$ to enhance the convergence to the local equilibrium.
We expect the result holds for $\kappa=0$, but our method is restricted to $\kappa>2^{-1}$, see, e.g., \eqref{eq:assp-k}.
Observe that it is \emph{not} the weakly asymmetric exclusion, since the symmetry is too weak to survive in the limit $n\to\infty$ under the hyperbolic scale.
From a macroscopic point of view, $n^{1+\kappa}L_\ss$ serves as vanishing viscosity when $\kappa<1$, see Section \ref{subsec:viscous} below.

The strengths of the boundary reservoirs are regulated by $n^\theta$.
In particular, we call the reservoirs \emph{weak} or slow if $\theta<0$ and \emph{strong} or fast if $\theta>0$.
The fast case has been treated in \cite{Xu21}, and the result is summarised in Theorem \ref{thm:fast}.
As mentioned in the introduction, we focus on two cases: (1) $\theta<0$ and (2) $\theta=0$, $\alpha(t)=\beta(t)\equiv0$.
A conjecture for the case $\theta=0$ with general $L^\infty$ functions $\alpha$, $\beta$, $\gamma$, $\delta$ can be found in Section \ref{subsec:viscous}.

\subsection{The macroscopic equation}
\label{subsec:pde}

Suppose that $\alpha$, $\beta$, $\gamma$, $\delta$ are non-negative, $L^\infty$ functions.
The density of the particles in $\eta(t)$ is expected to evolve with the initial--boundary problem of the scalar conservation law on $[0,1]$ given by
\begin{equation}
\label{eq:cl}
  \left\{
  \begin{aligned}
    &\,\partial_tu(t,x)+p\partial_x\big[J(u(t,x))\big]=0, \quad J(u)=u(1-u),\\
    &\,u(t,0)=v_-(t), \quad u(t,1)=v_+(t), \quad u(0,x)=v_0(x),
  \end{aligned}
  \right.
\end{equation}
where $v_\pm \in L^\infty(\bR_+)$ and $v_0 \in L^\infty([0,1])$ satisfy that $v_\pm$, $v_0\in[0,1]$.
These functions would be determined in Theorem \ref{thm:fast} and \ref{thm:slow}.
It is well-known that the weak solution to the nonlinear hyperbolic equation is not regular or unique.
In particular, even when $v_0$ and $v_\pm$ are smooth functions, discontinuities such as shocks and boundary layers appear within finite time.
We are hence forced to consider the entropy solution.

\begin{defn}
\label{defn:ent-flux}
A pair of functions $F$, $Q$ on $[0,1]$ is called a \emph{Lax entropy flux pair} associated to $J$, if they are twice continuously differentiable,
\begin{align}
  Q'(u)=J'(u)F'(u)=(1-2u)F'(u), \quad F''(u)\ge0, \quad \forall\,u\in[0,1].
\end{align}
A pair of functions $\mathcal F$, $\mathcal Q$ on $[0,1]^2$ is called a \emph{boundary entropy flux pair}, if $(\mathcal F,\mathcal Q)(\cdot,v)$ is a Lax entropy flux pair for each $v\in[0,1]$ and
\begin{align}
  \mathcal F(v,v)=\partial_u\mathcal F(u,v)|_{u=v}=\mathcal Q(v,v)=0, \quad \forall\,v\in[0,1].
\end{align}
\end{defn}

\begin{defn}
\label{defn:ent-sol}
A function $u \in L^\infty(\bR_+\times[0,1])$ is called the \emph{entropy solution} to \eqref{eq:cl} if and only if the following conditions are fulfilled \cite{Otto96}:
\begin{itemize}
\item[(\romannum1)] the entropy inequality: for all Lax entropy flux pairs $(F,Q)$,
\begin{align}
\label{eq:ent-sol-1}
  \partial_t\big[F(u)\big]+p\partial_x\big[Q(u)\big] \le 0 \quad \text{as a distribution on} \ \bR_+\times(0,1);
\end{align}
\item[(\romannum2)] the $L^1$-initial condition:
\begin{align}
\label{eq:ent-sol-2}
  \esslim_{t\to0+} \int_0^1 \big|u(t,x)-v_0(x)\big|dx = 0;
\end{align}
\item[(\romannum3)] the Otto type boundary conditions: for all boundary flux $\mathcal Q$ and all $\phi\in\cC(\bR_+)$, $\phi\ge0$ with compact support,
\begin{equation}
\label{eq:ent-sol-3}
  \begin{aligned}
    \esslim_{x\to0+} \int_0^\infty \phi(t)\mathcal Q\big(u(t,x),v_-(t)\big)dt \le 0,\\
    \esslim_{x\to1-} \int_0^\infty \phi(t)\mathcal Q\big(u(t,x),v_+(t)\big)dt \ge 0.
  \end{aligned}
\end{equation}
\end{itemize}
\end{defn}

\begin{rem}
\label{rem:viscous}
The concept of entropy solution is motivated by the viscous approximate.
For $\ve>0$, let $u^\ve$ be the strong solution to the parabolic perturbation of \eqref{eq:cl}:
\begin{align}
\label{eq:viscous}
  \partial_tu+p\partial_xJ=\ve\partial_x^2u, \quad u|_{x=0}=v_-, \quad u|_{x=1}=v_+, \quad u|_{t=0}=v_0.
\end{align}
As $\ve\to0$, $u^\ve$ converges, with respect to the topology of $\cC([0,T];L^1([0,1]))$, to the entropy solution $u$ defined in Definition \ref{defn:ent-sol} for any $T>0$, see \cite[Theorem 2.8.20]{MNRR96}.
Also, the entropy solution is unique, see \cite[Theorem 2.7.28]{MNRR96}.
\end{rem}

\begin{rem}
\label{rem:ent-sol}
The validity of \eqref{eq:ent-sol-1} for all Lax entropy flux pairs is equivalent to the validity for a countable set of Lax entropy flux pairs.
Indeed, let $F\in\cC^\infty(\bR)$ satisfy that
\begin{align}
  F(u)=|u|, \ \forall\,|u|>1, \quad F'(0)=0, \quad F''(u)\ge0, \ \forall\,u\in\bR.
\end{align}
For rational number $c$ and positive integer $m$, define
\begin{align}
  F_{m,c}(u):=m^{-1}F\big(m(u-c)\big), \quad \forall\,u\in[0,1].
\end{align}
Then, $F_{m,c}$ is a Lax entropy.
The corresponding flux can be chosen as
\begin{align}
  Q_{m,c}(u):=\int_c^u J'(w)F'(m(w-c))dw, \quad \forall\,u\in[0,1].
\end{align}
Suppose that \eqref{eq:ent-sol-1} holds for all pairs in the countable set $\{(F_{m,c},Q_{m,c});m\in\bN,c\in\bQ\}$.
Then, it is easy to verify \eqref{eq:ent-sol-1} for all $(F_c,Q_c)$, $c\in\bR$, where $F_c(u):=|u-c|$ and $Q_c(u):=\sgn(u-c)(J(u)-J(c))$.
The conclusion then follows, since every convex function belongs to the convex hull of the set of all affine functions and all $F_c$, $c\in\bR$.
\end{rem}

If $v_0$ and $v_\pm$ have bounded variations, then also does $u$.
In this case, the limits of $u(t,x)$ for $t\to0$, $x\to0$ and $x\to1$ are well-defined, and (\romannum3) is equivalent to the Bardos--LeRoux--N\'ed\'elec conditions \cite{BLN79}.
Under general $L^\infty$ framework, those limits may not exist.
Instead, Vasseur \cite{Vass01} proves that $u$ possesses strong initial and boundary traces.
This result is crucial for our argument, so we state it here.

\begin{prop}[{\cite[Theorem 1]{Vass01}}]
\label{prop:bd-trace}
Fix some arbitrary $T>0$.
Assume that $u\in L^\infty([0,T]\times[0,1])$ satisfies (\romannum1) in Definition \ref{defn:ent-sol} in sense of distributions on $(0,T)\times(0,1)$, then there are $u_0 \in L^\infty([0,1])$ and $u_\pm \in L^\infty([0,T])$, such that
\begin{gather}
  \esslim_{t\to0+} \int_0^1 |u(t,x)-u_0(x)|dx = 0,\\
  \esslim_{x\to0+} \int_0^T |u(t,x)-u_-(t)|dx = 0, \quad \esslim_{x\to1-} \int_0^T |u(t,x)-u_+(t)|dx = 0.
\end{gather}
\end{prop}

Using the boundary traces, \eqref{eq:ent-sol-2}, \eqref{eq:ent-sol-3} are equivalent to $u_0=v_0$, $\pm\mathcal Q(u_\pm,v_\pm)\ge0$.
We can further rewrite the boundary conditions explicitly as follows.

\begin{prop}
\label{prop:ent-sol}
Assume that $u \in L^\infty([0,T]\times[0,1])$ satisfies (\romannum1) in Definition \ref{defn:ent-sol}, then (\romannum3) holds if and only if for almost all $t\in[0,T]$,
\begin{equation}
\label{eq:ent-sol-4}
  \begin{aligned}
    u_-(t) &\in \big\{v_-(t)\big\}\cup \big[ 1-\min\{2^{-1},v_-(t)\},1],\\
    u_+(t) &\in \big[0,1-\max\{2^{-1},v_+(t)\}\big]\cup\big\{v_+(t)\big\}.
  \end{aligned}
\end{equation}
\end{prop}

\begin{proof}
First assume \eqref{eq:ent-sol-4}.
Note that for any boundary entropy pair $(\mathcal F,\mathcal Q)$, $\partial_w\mathcal F(w,v)\ge0$ if $w>v$, and $\partial_w\mathcal F(w,v)\le0$ if $w \le v$.
If $v_->2^{-1}$, \eqref{eq:ent-sol-4} means that $u_-\in[2^{-1},1]$ and
\begin{align}
  \mathcal Q(u_-,v_-)=\int_{v_-}^{u_-} (1-2w)\partial_w\mathcal F(w,v_-)dw\le0.
\end{align}
Meanwhile, if $v_- \le 2^{-1}$, then \eqref{eq:ent-sol-4} requires that $u_-=v_-$ or $u_-\in[1-v_-,1]$.
Observe that when $u_-=v_-$, $\mathcal Q(u_-,v_-)=0$, and when $u_-\in[1-v_-,1]$,
\begin{equation}
  \begin{aligned}
    \mathcal Q(u_-,v_-) &\le \int_{v_-}^{1-v_-} (1-2w)\partial_w\mathcal F(w,v_-)dw\\
    &= \int_{v_-}^{\frac12} (2w-1)\int_w^{1-w} \partial_{w'}^2\mathcal F(w',w)dw'dw\le0.
  \end{aligned}
\end{equation}
Therefore, $\mathcal Q(u_-,v_-)\le0$.
The assertion for $u_+$ follows similarly.

On the other hand, assume that $\mathcal Q(u_-,v_-)\le0$, $\mathcal Q(u_+,v_+)\ge0$ for all boundary entropy flux pair $(\mathcal F,\mathcal Q)$.
For each $m\in\bN$, let $f_m\in\cC^1([0,1]\times\bR)$ be such that $\partial_uf_m(u,v)\ge0$,
\begin{align}
  f_m(u,v)=
  \begin{cases}
    -1, &\text{if} \ u\le\underline v_m-m^{-1},\\
    0, &\text{if} \ u\in[\underline v_m, \bar v_m],\\
    1, &\text{if} \ u\ge\bar v_m+m^{-1},
  \end{cases}
  \qquad
  \begin{aligned}
    &\underline v_m:=\min\{2^{-1},v\}-m^{-1},\\
    &\bar v_m:=v+m^{-1}.
  \end{aligned}
\end{align}
Then, $\mathcal F_m(u,v):=\int_v^u f_m(w,v)dw$ is a boundary entropy corresponding to the flux $\mathcal Q_m(u,v)=\int_v^u (1-2w)f_m(w,v)dw$.
Note that $\mathcal Q_m(u,v)=0$ if $\underline v_m \le u \le \bar v_m$.
If $u < \underline v_m-m^{-1} < 2^{-1}$,
\begin{equation}
  \begin{aligned}
    \mathcal Q_m(u,v) &= -\int_u^{\underline v_m-\frac1m} -(1-2w)dw - \int_{\underline v_m-\frac1m}^{\underline v_m} (1-2w)f_m(w,v)dw\\
    &\ge \int_u^{\underline v_m-\frac1m} (1-2w)dw = J \left( \underline v_m-\frac1m \right) - J(u) > 0.
  \end{aligned}
\end{equation}
Similarly, if $v<2^{-1}$, then $\bar v_m+m^{-1}<2^{-1}$ for $m$ large enough.
In this case, for $u$ such that $\bar v_m+m^{-1}<u<1-(\bar v_m+m^{-1})$,
\begin{equation}
  \begin{aligned}
    \mathcal Q_m(u,v) &= \int_{\bar v_m}^{\bar v_m+\frac1m} (1-2w)f_m(w,v)dw + \int_{\bar v_m+\frac1m}^u (1-2w)dw\\
    &\ge \int_{\bar v_m+\frac1m}^u (1-2w)dw = J(u) - J \left( \bar v_m+\frac1m \right) > 0.
  \end{aligned}
\end{equation}
Since $\mathcal Q_m(u_-,v_-)\le0$ for all $m$, the first formula in \eqref{eq:ent-sol-4} follows.
Repeating the argument for $\mathcal Q_m(u_+,v_+)$, the second formula holds similarly.
\end{proof}

\subsection{Hydrodynamic limit}
\label{subsec:hl}

Recall that $\eta=\eta(t)$ is the process generated by $L_{n,t}$ in \eqref{eq:generator} and the initial distribution $\mu_{n,0}$.
Assume that there exists $v_0 \in L^\infty([0,1])$, such that
\begin{align}
\label{eq:assp-initial}
  \lim_{n\to\infty} \mu_{n,0} \left\{ \left| \frac1n\sum_{i=1}^n \eta_i(0)\vf \left( \frac in \right) - \int_0^1 v_0(x)\vf(x)dx \right| > \ve \right\} = 0,
\end{align}
for any $\ve>0$ and any $\vf \in \cC([0,1])$.

Denote by $\bP_n$ the distribution of $\eta(\cdot)$ on the path space.
The corresponding expectation is written as $\bE_n$.
Our aim is to prove the hydrodynamic limit: for any $\ell\in\bN$, local observation $f=f(\eta_1,\ldots,\eta_\ell)$, test function $\psi\in\cC_c(\bR^2)$ and $\ve>0$,
\begin{equation}
\label{eq:hl}
  \begin{aligned}
    \lim_{n\to\infty} \bP_n \Bigg\{ \bigg| \int_0^\infty &\frac1n\sum_{i=0}^{n-\ell} f\big(\tau_i\eta(t)\big)\psi \left( t,\frac in \right) dt\\
    &- \int_0^\infty \!\! \int_0^1 \langle f \rangle(u(t,x))\psi(t,x)\,dx\,dt \bigg| > \ve \Bigg\} = 0,
  \end{aligned}
\end{equation}
where $f(\tau_i\eta):=f(\eta_{i+1},\ldots,\eta_{i+\ell})$, $\langle f \rangle(\rho)$ is the expectation of $f$ with respect to the Bernoulli measure with density $\rho\in[0,1]$, and $u$ is decided in Theorem \ref{thm:fast} and \ref{thm:slow} below.

When $\theta>0$, the following result is proved in \cite{Xu21}.

\begin{thm}[Strong reservoirs]
\label{thm:fast}
Suppose that $\kappa\in(5/7,1)$, $\theta>0$ and $\essinf\{\alpha,\beta,\gamma,\delta\}>0$, then \eqref{eq:hl} holds with the entropy solution $u$ to \eqref{eq:cl} with $v_-(t)=\alpha(t)(\alpha(t)+\gamma(t))^{-1}$ and $v_+(t)=\delta(t)(\beta(t)+\delta(t))^{-1}$.
\end{thm}

Hereafter we focus on the regime $\theta\le0$ and prove the following result.

\begin{thm}
\label{thm:slow}
Suppose that $\kappa\in(1/2,1)$ and one of the following holds: (1) $\theta<0$; (2) $\theta=0$, $\alpha(t)=\beta(t)=0$ for almost all $t>0$.
Then \eqref{eq:hl} holds with the entropy solution $u$ to \eqref{eq:cl} with $v_-(t)=0$ and $v_+(t)=1$.
\end{thm}

Theorem \ref{thm:slow} is proved in Section \ref{subsec:proof}, employing the convergence to the local equilibrium in Section \ref{sec:local-equil}, the compensated compactness argument in Section \ref{sec:com-com} and the computation of initial--boundary traces in Section \ref{sec:ent-sol}.
The proofs for the two cases (1) and (2) are mostly in common, hence are stated together.

\begin{rem}
\label{rem:slow}
From Proposition \ref{prop:ent-sol}, the boundary conditions $v_-=0$, $v_+=1$ essentially mean that the boundary trace $u_\pm(t)\in\{0,1\}$ for almost all $t>0$.
Heuristically speaking, $u_+(t)=1$ is the natural result of that particles are not allowed to escape from the right boundary effectively in both cases.
Meanwhile, the possible discontinuity can be understood in the following way.
Suppose that the process starts from an initial configuration empty in a macroscopic block $\{[an],\ldots,n\}$ for some $a<1$.
Due to the rightward drift, the particles pumped in site $n$ by the reservoir would be restricted within a small distance from $n$.
Thus, $u_+(t)=0$ would be observed until the drift managed to transport a macroscopically visible number of particles into this block.
Similar phenomenon is exhibited near $x=0$, i.e., $u_-(t)=1$.
\end{rem}

\begin{rem}
A special case treated in Theorem \ref{thm:slow} is $\alpha$, $\beta$, $\gamma$, $\delta\equiv0$.
This defines an \emph{isolated system} in contact with no reservoirs.
As no particle exchange happens between the system and the surroundings, the total number of particles is preserved globally.
\end{rem}

\begin{rem}
\label{rem:separ}
One can generalise the model in Theorem \ref{thm:fast} by regulating the birth and death rates at the boundaries separately.
Precisely, for $\theta_{\pm,j}\in\bR$, $j=1$, $2$, define
\begin{equation}
  \begin{aligned}
    L_{-,t}f(\eta) &:= \big[n^{\theta_{-,1}}\alpha(t)(1-\eta_1)+n^{\theta_{-,2}}\gamma(t)\eta_1\big]\big[f(\eta^1)-f(\eta)\big],\\
    L_{+,t}f(\eta) &:= \big[n^{\theta_{+,1}}\beta(t)\eta_n+n^{\theta_{+,2}}\delta(t)(1-\eta_n)\big]\big[f(\eta^n)-f(\eta)\big].
  \end{aligned}
\end{equation}
Let $\eta(\cdot)$ be the Markov process generated by $n(pL_\ta+\sigma n^\kappa L_\ss+L_{-,t}+L_{+,t})$.
The result in Theorem \ref{thm:slow} extends to the case $\theta_{\pm,1}<0$, $\theta_{\pm,2}\le0$ with slight modifications on the proof, see Remark \ref{rem:separ1} and \ref{rem:separ2} for details.
\end{rem}

\subsection{Vanishing viscosity limit}
\label{subsec:viscous}

Recall the viscous approximate of the entropy solution in Remark \ref{rem:viscous}.
In this subsection, we introduce another type of viscous limit arising more naturally from the microscopic dynamics.

To simplify the problem, suppose that $\alpha$, $\beta$, $\gamma$ and $\delta$ are constant in time, and the initial distribution $\mu_{n,0}$ subjects to some smooth profile $v_0\in\cC^\infty([0,1])$, $0 \le v_0 \le 1$.
Recall that $\kappa<1$.
For small but fixed $\ve>0$, choose $\sigma=\sigma_n=\ve n^{1-\kappa}$ in \eqref{eq:generator} to define
\begin{align}
  L'_n = n^2\big[\ve L_\ss + pn^{-1}L_\ta + n^{\theta_*}(L_-+L_+)\big], \quad \theta_*:=\theta-1.
\end{align}
Here $n^2$ corresponds to the diffusive time scale, under which $L_\ss$ gives non-vanishing viscosity.
$L'_n$ then generates the weakly asymmetric simple exclusion (WASEP) in contact with reservoirs.
The corresponding hydrodynamic equation is obtained in \cite{CapiG21}:
\begin{align}
\label{eq:hl-wasep}
  \partial_tu(t,x)=\ve\partial_x^2u(t,x)-p\partial_xJ(u(t,x)), \quad u(0,x)=v_0(x),
\end{align}
for $(t,x)\in(0,T)\times(0,1)$, with the boundary conditions given by
\begin{align}
  \label{eq:bd-wasep-fast}
  &\text{if} \ \theta_*=\theta-1>-1, \quad u|_{x=0}=\alpha(\alpha+\gamma)^{-1}, \ u|_{x=1}=\delta(\beta+\delta)^{-1};\\
  \label{eq:bd-wasep}
  &\text{if} \ \theta_*=\theta-1=-1, \quad
  \begin{cases}
    [\ve\partial_xu-pJ(u)+\alpha-(\alpha+\gamma)u]\,\big|_{x=0}=0,\\
    [\ve\partial_xu-pJ(u)-\delta+(\beta+\delta)u]\,\big|_{x=1}=0;\\
  \end{cases}\\
  \label{eq:bd-wasep-slow}
  &\text{if} \ \theta_*=\theta-1<-1, \quad
  \begin{cases}
    [\ve\partial_xu-pJ(u)]\,\big|_{x=0}=0,\\
    [\ve\partial_xu-pJ(u)]\,\big|_{x=1}=0.
  \end{cases}
\end{align}
In Remark \ref{rem:viscous}, we see that the vanishing viscosity limit $\ve\to0$ of \eqref{eq:hl-wasep}--\eqref{eq:bd-wasep-fast} gives the hydrodynamic equation in Theorem \ref{thm:fast}.
Similar convergence is proved in below for the boundary conditions \eqref{eq:bd-wasep} and \eqref{eq:bd-wasep-slow}, respectively.

\begin{thm}
\label{thm:viscous}
Let $u^\ve\in\cC^\infty([0,T]\times[0,1])$ solve \eqref{eq:hl-wasep} and \eqref{eq:bd-wasep}.
As $\ve\to0$, $u^\ve$ converges in the weak-$\star$ topology of $L^\infty([0,T]\times[0,1])$ to the entropy solution to \eqref{eq:cl} with
\begin{equation}
\label{eq:bd-asep}
  \begin{aligned}
    &v_-=\frac{p+\alpha+\gamma-\sqrt{(p-\alpha+\gamma)^2+4\alpha\gamma}}{2p},\\
    &v_+=\frac{p-\beta-\delta+\sqrt{(p-\beta+\delta)^2+4\beta\delta}}{2p}.
  \end{aligned}
\end{equation}
For the solution to \eqref{eq:hl-wasep} and \eqref{eq:bd-wasep-slow}, the same convergence holds, and the limit is given by \eqref{eq:cl} with $v_-=0$, $v_+=1$, i.e., the hydrodynamic equation in Theorem \ref{thm:slow}.
\end{thm}

\begin{proof}
We prove only for \eqref{eq:bd-wasep}.
For \eqref{eq:bd-wasep-slow} it suffices to take $\alpha$, $\beta$, $\gamma$ and $\delta=0$ in the result.
First, since $v_0(x)\in[0,1]$ and $v_\pm\in[0,1]$, standard argument shows that $u^\ve\in[0,1]$.
Observe also that $u^\ve$ satisfies the $H_1$-type estimate
\begin{align}
\label{eq:v1}
  \ve\int_0^T \!\! \int_0^1 \big(\partial_xu^\ve(t,x)\big)^2dx\,dt \le C, \quad \forall\,\ve>0.
\end{align}
Indeed, let $G$ be a primitive of $uJ'$: $G'(u)=uJ'(u)$.
Multiplying \eqref{eq:hl-wasep} by $u^\ve$,
\begin{align}
  \ve u^\ve\partial_x^2u^\ve=\frac12\partial_t \big(u^\ve\big)^2+p\partial_xG\big(u^\ve\big).
\end{align}
As $u^\ve$, $G(u^\ve)$ are bounded, the integration on $[0,T]\times[0,1]$ yields that
\begin{align}
  \left| \int_0^T \big(\ve u^\ve\partial_xu^\ve\big)\Big|_{x=0}^{x=1}\,dt-\int_0^T \!\! \int_0^1 \big(\partial_xu^\ve\big)^2dx\,dt \right| \le C.
\end{align}
With \eqref{eq:bd-wasep}, the first term above is bounded, then we obtain \eqref{eq:v1}.

Next, by the compensated compactness argument (see, e.g., \cite[Section 5.D]{Evans90}), $u^\ve$ converges, along proper subsequence, to some $u \in L^\infty([0,T]\times[0,1])$ in the weak-$\star$ topology.
It is also straightforward to verify conditions (\romannum1--\romannum2) in Definition \ref{defn:ent-sol} for $u$.

Hereafter, we focus on the boundary conditions and show that $u$ satisfies (\romannum3) in Definition \ref{defn:ent-sol} with $v_\pm$ in \eqref{eq:bd-asep}.
For any boundary entropy flux pair $(\mathcal F,\mathcal Q)$, we claim that
\begin{align}
\label{eq:v2}
  \partial_u\mathcal F(u,v_-)\mathfrak j(u)-p\mathcal Q(u,v_-)\ge0, \quad u\in[0,1],
\end{align}
where $\mathfrak j(u):=pJ(u)+(\alpha+\gamma)u-\alpha$.
To see this, denote $(f,q)(u)=(\mathcal F,\mathcal Q)(u,v_-)$, then $f$ is convex, $f(v_-)=f'(v_-)=0$ and $q(v_-)=0$, $q'=J'f'$, so that $(f'\mathfrak j-pq)(v_-)=0$.
Since $f'\mathfrak j'=f'(pJ'+\alpha+\gamma)=pq'+(\alpha+\gamma)f'$,
\begin{align}
  (f'\mathfrak j-pq)' = f''\mathfrak j+f'\mathfrak j'-pq' =f''\mathfrak j+(\alpha+\gamma)f'.
\end{align}
For $u\in[0,v_-)$, observe that $\mathfrak j(u)<0$ and $f'(u)<0$, thus $(f'\mathfrak j-pq)'(u)<0$.
Similarly, $(f'\mathfrak j-pq)'(u)>0$ for $u\in(v_-,1]$.
The claim then follows.

For a smooth function $\psi\ge0$ such that $\psi(0,x)=\psi(T,x)=\psi(t,1)=0$,
\begin{align}
  \int_0^T \!\! \int_0^1 f(u^\ve)\partial_t\psi\,dx\,dt = \int_0^T \!\! \int_0^1 f'(u^\ve)\big[p\partial_xJ(u^\ve)-\ve\partial_x^2u^\ve\big]\psi\,dx\,dt.
\end{align}
Since $f'(u)\partial_x^2u=\partial_x^2[f(u)]-f''(u)(\partial_xu)^2\le\partial_x^2[f(u)]$ and $f'J'=q'$,
\begin{align}
  \int_0^T \!\! \int_0^1 f(u^\ve)\partial_t\psi\,dx\,dt \ge \int_0^T \!\! \int_0^1 \big[p\partial_xq(u^\ve)-\ve\partial_x^2f(u^\ve)\big]\psi\,dx\,dt.
\end{align}
Performing integration by parts and taking \eqref{eq:bd-wasep} into consideration,
\begin{equation}
  \begin{aligned}
    \int_0^T \!\! \int_0^1 f(u^\ve)\partial_t\psi\,dx\,dt \ge &\int_0^T \!\! \int_0^1 \big[\ve\partial_xf(u^\ve)-pq(u^\ve)\big]\partial_x\psi\,dx\,dt\\
    &+ \int_0^T \big[f'(u^\ve)\mathfrak j(u^\ve)-pq(u^\ve)\big]\psi\,\big|_{x=0}\,dt.
  \end{aligned}
\end{equation}
Let $\ve\to0$ and notice that the integral of $\ve\partial_xf(u^\ve)\partial_x\psi$ vanishes due to \eqref{eq:v1}, while the last line is non-negative because of \eqref{eq:v2}.
Hence, the limit $u$ satisfies that
\begin{align}
\label{eq:v3}
  \int_0^T \!\! \int_0^1 \big[f(u)\partial_t\psi + pq(u)\partial_x\psi\big]\,dx\,dt \ge 0, \quad (f,q)=(\mathcal F,\mathcal Q)(\cdot,v_-).
\end{align}
Recall that $u$ satisfies the entropy inequality, Proposition \ref{prop:bd-trace} yields the existence of some $u_-$, such that $\|u(\cdot,x)-u_-\|_{L^1}$ as $x\to0$.
Note that $\psi$ does not vanish at $x=0$, \eqref{eq:v3} then means that $q(u_-)=\mathcal Q(u_-,v_-)\le0$ almost surely on $[0,T]$.
Similar argument works for $\mathcal Q(u_+,v_+)$, so that the conclusion follows.
\end{proof}

From \eqref{eq:bd-wasep} and Theorem \ref{thm:viscous}, the following conjecture is not a surprise.

\begin{conj}
\label{conj:critical}
For the case $\theta=0$, the hydrodynamic equation is given by \eqref{eq:cl} with $\nu_\pm=\nu_\pm(t)$ determined through \eqref{eq:bd-asep} with $(\alpha,\beta,\gamma,\delta)=(\alpha,\beta,\gamma,\delta)(t)$.
\end{conj}

\begin{rem}
\label{rem:liggett}
Besides the case (2) in Theorem \ref{thm:slow}, the conjecture is also proved for another choice of parameters: $\kappa=\theta=0$, $(\alpha,\beta,\gamma,\delta)$ does not evolve in time, and the following relation \cite{Liggett75,DEHP93,Baha12} is satisfied:
\begin{align}
\label{eq:liggett}
    \frac\alpha{p+\sigma} + \frac\gamma\sigma = 1, \quad \frac\beta{p+\sigma} + \frac\delta\sigma = 1.
\end{align}
Under \eqref{eq:liggett}, $v_\pm$ in \eqref{eq:bd-asep} are simply given by
\begin{align}
\tag{\ref{eq:bd-asep}'}
  v_-=\frac\alpha{p+\sigma}=1-\frac\gamma\sigma, \quad v_+=1-\frac\beta{p+\sigma}=\frac\delta\sigma.
\end{align}
Consider the ASEP on infinite lattice $\bZ$ with jump rate $p+\sigma$ to the right and $\sigma$ to the left.
Suppose that $\{\eta_i;i<1\}$, $\{\eta_i;i>n\}$ are respectively distributed according to Bernoulli measure with densities $v_-$ and $v_+$.
The dynamics with reservoirs defined through \eqref{eq:liggett} is the projection of this finite dynamics on $\{1,\ldots,n\}$.
Employing this observation, the hydrodynamic limit is proved in \cite{Baha12}.
\end{rem}

\subsection{Stationary solution}
\label{subsec:stat-sol}

In this subsection, we present some facts about the asymptotic solution $u_\infty := \lim_{t\to\infty} u(t,\cdot)$, where $u$ is the entropy solution to the hydrodynamic equation \eqref{eq:cl} with $v_-=0$, $v_+=1$.

First, observe that the stationary entropy solution $u(t,x)=\bar u(x)$ is not unique.
Indeed, they are given by the class of profiles with single upward shock:
\begin{align}
\label{eq:stat-ent-sol}
  \big\{\bar u_y \in L^\infty([0,1]); \ y\in[0,1]\big\}, \quad \bar u_y(x) := \mathbf1_{(y,1)}(x).
\end{align}
To investigate the limit, note that Dynkin's formula yields that for all $t>0$,
\begin{align}
\label{eq:stat1}
  \bE_n \left[ \frac1n\sum_{i=1}^n \eta_i(t) - \frac1n\sum_{i=1}^n \eta_i(0) - \int_0^t \frac1n\sum_{i=1}^n L_{n,s} [\eta_i(s)]ds \right] = 0.
\end{align}
Recalling \eqref{eq:generator}, the spatial average of $L_{n,t} [\eta_i]$ reads
\begin{align}
\label{eq:stat2}
  \frac1n\sum_{i=1}^n L_{n,t} [\eta_i] = n^\theta\Big[\alpha(t)-\big(\alpha(t)+\gamma(t)\big)\eta_1-\big(\beta(t)+\delta(t)\big)\eta_n+\delta(t)\Big].
\end{align}
For any fixed $T>0$, integrating \eqref{eq:stat1} in time and inserting \eqref{eq:stat2},
\begin{align}
  \bE_n \left[ \int_0^T \frac1n\sum_{i=1}^n \eta_i(t)dt \right] = T \times \bE_n \left[ \frac1n\sum_{i=1}^n \eta_i(0) \right] + O(n^\theta).
\end{align}
Consider the case $\theta<0$ and take the limit $n\to\infty$.
By \eqref{eq:assp-initial}, the right-hand side converges to $T\int v_0(x)dx$.
Therefore, by taking $f=\eta_1$ and\footnote{Here $\psi\not\in \cC_c(\bR^2)$, but it can be easily approximated by continuous functions.} $\psi\equiv\mathbf1_{[0,T]}$ in \eqref{eq:hl},
\begin{align}
  \int_0^T \!\! \int_0^1 u(t,x)dx\,dt = \lim_{n\to\infty} \bE_n \left[ \int_0^T \frac1n\sum_{i=1}^n \eta_i(s)ds \right] = T\int_0^1 v_0(x)dx.
\end{align}
In conclusion, $u$ satisfies strictly the conservation law:
\begin{align}
\label{eq:cl-amount}
  \int_0^1 u(t,x)dx = \int_0^1 v_0(x)dx := \mathfrak m, \quad \forall\,t>0.
\end{align}
From \eqref{eq:stat-ent-sol} and \eqref{eq:cl-amount}, it is reasonable to predict that $u_\infty=\bar u_{1-\mathfrak m}$.
This would be justified if one can prove that under the weak-$\star$ topology, $u(t,\cdot)$ converges to some $\bar u \in L^\infty([0,1])$.

\begin{rem}
For other boundary data, \eqref{eq:cl-amount} fails in general due to the discontinuities at the boundaries shown in \eqref{eq:ent-sol-4}.
\end{rem}

\subsection{Proof of Theorem \ref{thm:slow}}
\label{subsec:proof}

For $k \le n$, define the uniform average
\begin{align}
\label{eq:bareta}
  \bar\eta_{i,k}:=\frac1k\sum_{i'=0}^{k-1} \eta_{i-i'}, \quad \forall\,i=k, k+1, \ldots, n.
\end{align}
For $i=k$, ..., $n-k+1$, define the smoothly weighted average
\begin{align}
\label{eq:hateta}
  \hat\eta_{i,k}:=\frac1k\sum_{i'=0}^{k-1} \bar\eta_{i+i',k} = \sum_{i'=-k+1}^{k-1} w_{i'}\eta_{i-i'}, \quad w_{i'}=\frac{k-|i'|}{k^2}.
\end{align}
In the rest of this article, $k$ is chosen to be the mesoscopic scale such that
\begin{align}
\label{eq:assp-k}
  k = k_n = \big[n^{\kappa'}\big], \quad \kappa' \in \left( \min \left\{ 2\kappa-1,\frac{1+\kappa}3 \right\}, \kappa \right).
\end{align}
It is well defined since we require $\kappa\in(2^{-1},1)$ in Theorem \ref{thm:slow}.
Observe from this definition that $k \ll n^{(1+\kappa)/2}$ and $k \gg \sqrt n$.
These conditions are necessary in Proposition \ref{prop:local-equil} and Section \ref{sec:com-com}.
Now we state the proof of Theorem \ref{thm:slow}.

\begin{proof}[Proof of Theorem \ref{thm:slow}]
By the local equilibrium proved in Proposition \ref{prop:local-equil}, to verify \eqref{eq:hl} it suffices to show for $g:=\langle f \rangle$ and $\psi\in\cC_c(\bR^2)$ that
\begin{align}
  \lim_{n\to\infty} \bE_n \left[ \int_0^\infty \frac1n\sum_{i=k+1}^{n-k} g\big(\hat\eta_{i,k}\big)\psi \left( t,\frac in \right) dt \right] = \int_0^\infty \!\! \int_0^1 g(u)\psi\,dx\,dt.
\end{align}
Fix some arbitrary $T>0$ and denote $\Sigma_T=[0,T]\times[0,1]$.
Proposition \ref{prop:hl1} yields that there is a probability measure $\fQ$ on $L^\infty(\Sigma_T)$, such that along proper subsequence,
\begin{align}
\label{eq:hl1}
  \lim_{n\to\infty} \bE_n \left[ \int_0^T \frac1n\sum_{i=k+1}^{n-k} g\big(\hat\eta_{i,k}\big)\psi \left( t,\frac in \right) dt \right] = E^\fQ \left[ \iint_{\Sigma_T} g(\rho)\psi\,dx\,dt \right].
\end{align}

We want to show that $\fQ$ concentrates on the entropy solution to \eqref{eq:cl} with $v_-=0$, $v_+=1$.
To this end, first obtain the entropy inequality from Lemma \ref{lem:ent-ineq}:
\begin{equation}
\label{eq:p1}
  \fQ \left\{ \partial_t\big[f(\rho)\big] + p\partial_x\big[q(\rho)\big]\le0
  \ \bigg|\ 
  \begin{aligned}
    &\text{all Lax entropy}\\
    &\text{flux pair }(f,q)
  \end{aligned}
  \right\} =1,
\end{equation}
where the inequality holds in sense of distributions on $\Sigma_T$.
By Proposition \ref{prop:bd-trace}, this implies that $\rho$ has initial and boundary traces $\rho_0 \in L^\infty([0,1])$, $\rho_\pm \in L^\infty([0,T])$:
\begin{align}
\label{eq:bd-traces1}
  \esslim_{t\to0+} \rho(t,\cdot)=\rho_0,\quad \esslim_{x\to0+} \rho(\cdot,x)=\rho_-, \quad \esslim_{x\to1-} \rho(\cdot,x)=\rho_+,
\end{align}
where the limits hold $\fQ$-almost surely in $L^1$ topology.
In view of Definition \ref{defn:ent-sol} and Remark \ref{rem:slow}, it suffices to prove that $\rho_0=v_0$ and $\rho_\pm\in\{0,1\}$, $\fQ$-almost surely.

By taking $(f,q)=(\mathrm{id}, J)$ and $(f,q)=-(\mathrm{id},J)$ in \eqref{eq:p1}, we have $\fQ$-almost surely that, for any $\psi\in\cC(\Sigma_T)$ which vanishes at $\partial\Sigma_T$,
\begin{align}
\label{eq:p2}
  \iint_{\Sigma_T} \big[\rho(t,x)\partial_t\psi(t,x) + J\big(u(t,x)\big)\partial_x\psi(t,x)\big]dx\,dt = 0.
\end{align}
Since the initial trace $\rho_0$ exists, \eqref{eq:p2} together with Lemma \ref{lem:weak-sol} yields that $\rho_0=v_0$.
Similarly, we have from Lemma \ref{lem:weak-bd} that $0=\esslim_{x\to0} J(\rho)=J(\rho_-)$.
As $J=u(1-u)$, it follows that $\rho_-\in\{0,1\}$.
The condition for $\rho_+$ holds similarly.

In consequence, $\rho$ is the entropy solution on $[0,T]\times[0,1]$, $\fQ$-almost surely.
Finally, one can extend the result to $\bR_+\times[0,1]$ using the uniqueness of entropy solution.
\end{proof}

\section{Local equilibrium}
\label{sec:local-equil}

Recall the dynamics $\eta=\eta(t)$ generated by $L_{n,t}$ in \eqref{eq:generator}.
In this section, we prove the convergence of $\eta(t)$ to its local equilibrium, stated as the following proposition.

\begin{prop}
\label{prop:local-equil}
If $1 \ll k \ll n^{(1+\kappa)/2}$, then
\begin{align}
\label{eq:local-equil}
  \lim_{n\to\infty} \bE_n \left[ \int_0^\infty \frac1n \left| \sum_{i=0}^{n-\ell} f(\tau_i\eta)\psi \left( t,\frac in \right) - \sum_{i=k+1}^{n-k} \langle f \rangle\big(\hat\eta_{i,k}\big)\psi \left( t,\frac in \right) \right| dt \right]=0,
\end{align}
for any local observation $f=f(\eta_1,\ldots,\eta_\ell)$ and $\psi\in\cC_c(\bR^2)$.
\end{prop}

\begin{proof}
For $f=f(\eta_1,\ldots,\eta_\ell)$ and $i=k$, ..., $n-k$, define
\begin{align}
  f_{i,k}=f_{i,k}(\eta_{i-k+1},\ldots,\eta_{i+k-1}):=\frac1{2k-\ell}\sum_{i'=-k}^{k-\ell-1} f(\tau_{i+i'}\eta).
\end{align}
We can substitute $f(\tau_i\eta)$ to $f_{i,k}$ in \eqref{eq:local-equil} with an error uniformly bounded by $Ckn^{-1}$.

Denote $g(\rho):=\langle f \rangle(\rho)$.
Without loss of generality, suppose for some $t_0>0$ that $\psi(t,\cdot)\equiv0$ for $t>t_0$.
Applying \CS inequality,
\begin{equation}
\label{eq:local-equil1}
  \begin{aligned}
    &\bE_n \left[ \int_0^\infty \frac1n \left| \sum_{i=k+1}^{n-k} \big[f_{i,k}-g\big(\hat\eta_{i,k}\big)\big]\psi \left( t,\frac in \right) \right| dt \right]^2\\
    \le\;&\frac1{n^2}\,\bE_n \left[ \int_0^{t_0} \sum_{i=k+1}^{n-k} \big|f_{i,k}-g\big(\hat\eta_{i,k}\big)\big|^2dt \right] \int_0^{t_0} \sum_{i=k+1}^{n-k} \psi \left(t, \frac in \right)^2dt.
  \end{aligned}
\end{equation}
Let $\mu_{n,t}$ be the distribution of $\eta(t)$.
For a probability measure $\mu$ on $\Omega_n$, define the Dirichlet form associated to the symmetric exclusion dynamics as
\begin{align}
\label{eq:dir}
  \fD_n(\mu):=\frac12\sum_{\eta\in\Omega_n} \sum_{i=1}^{n-1} \left( \sqrt{\mu(\eta^{i,i+1})}-\sqrt{\mu(\eta)} \right)^2.
\end{align}
By Lemma \ref{lem:block} and \ref{lem:dir} below, the right-hand side in \eqref{eq:local-equil1} is bounded from above by
\begin{align}
  \frac Cn\int_0^{t_0} \left[ k^2\fD_n(\mu_{n,t}) + \frac nk \right]dt \int_0^{t_0} \!\! \int_0^1 \psi^2dx\,dt \le C_\psi \left( \frac{k^2}{n^{1+\kappa}}+\frac1k \right).
\end{align}
The conclusion then follows.
\end{proof}

\begin{lem}
\label{lem:block}
For any probability measure $\mu$ on $\Omega_n$ and $f=f(\eta_1,\ldots,\eta_\ell)$,
\begin{align}
  \sum_{\eta\in\Omega_n} \sum_{i=k}^{n-k} \big|f_{i,k}-\langle f \rangle\big(\hat\eta_{i,k}\big)\big|^2\mu(\eta) \le C \left[ k^2\fD_n(\mu)+\frac nk \right].
\end{align}
\end{lem}

\begin{proof}
Define $I_k := \{i/(2k-1);i=0,1,\ldots,2k-1\}$.
For $\rho_* \in I_k$, let
\begin{align}
\Omega_{k,\rho_*} := \left\{\eta=(\eta_1,\eta_2,\dots,\eta_{2k-1}) \in \Omega_{2k-1}~\bigg|~\sum_{i=1}^{2k-1} \frac{\eta_i}{2k-1} = \rho_*\right\}.
\end{align}
Let $\mu_{i,k}$ be the projection of $\mu$ on $\{\eta_{i-k+1},\ldots,\eta_{i+k-1}\}$.
For $\rho_* \in I_k$, define $\bar\mu_{i,k}(\rho_*):=\mu_{i,k}(\Omega_{k,\rho_*})$ and $\mu_{i,k}(\cdot\,|\,\rho_*):=\mu_{i,k}(\cdot\,|\,\Omega_{k,\rho_*})$.
Also let $\nu_k(\,\cdot\,|\rho_*)$ be the uniform measure on $\Omega_{k,\rho_*}$.
With $g=\langle f \rangle$, $|f_{i,k}-\langle f \rangle(\hat\eta_{i,k})|^2$ is bounded from above by
\begin{align}
\label{eq:block0}
   2\big|f_{i,k}-g(\bar\eta_{i+k-1,2k-1})\big|^2+2\big|g(\bar\eta_{i+k-1,2k-1})-g\big(\hat\eta_{i,k}\big)\big|^2,
\end{align}
where $\bar\eta_{i,k}$ is defined in \eqref{eq:bareta}.

We estimate the two terms in \eqref{eq:block0} respectively.
For the first term, note that
\begin{equation}
\label{eq:block1}
  \begin{aligned}
    \mathcal I_i &:= \sum_{\eta\in\Omega_n} \big|f_{i,k}-g(\bar\eta_{i+k-1,2k-1})\big|^2\mu(\eta)\\
    &= \sum_{\eta\in\Omega_{2k-1}} \mathcal F^2\mu_{i,k}(\eta) = \sum_{\rho_* \in I_k} \bar\mu_{i,k}(\rho_*)\sum_{\eta\in\Omega_{k,\rho_*}} \mathcal F^2\mu_{i,k}(\eta|\rho_*),
  \end{aligned}
\end{equation}
where $\mathcal F=\mathcal F(\eta_1,\ldots,\eta_{2k-1})$ is given by
\begin{align}
  \mathcal F := \frac1{2k-\ell}\sum_{i=0}^{2k-\ell} f(\tau_i\eta) - g \left( \frac1{2k-1}\sum_{i=1}^{2k-1} \eta_i \right).
\end{align}
By the relative entropy inequality, for all $\rho_* \in I_k$ and $a>0$,
\begin{equation}
\label{eq:block2}
  \begin{aligned}
    \sum_{\eta\in\Omega_{k,\rho_*}} \mathcal F^2\mu_{i,k}(\eta|\rho_*) \le \frac{H_{i,k}(\rho_*)}a + \frac1a\log\sum_{\eta\in\Omega_{k,\rho_*}} e^{a\mathcal F^2}\nu_k(\eta|\rho_*),
\end{aligned}
\end{equation}
where the relative entropy $H_{i,k}$ is defined as
\begin{align}
\label{eq:rel-ent}
  H_{i,k}(\rho_*) := \sum_{\eta\in\Omega_{k,\rho_*}} \log \left[ \frac{\mu_{i,k}(\eta|\rho_*)}{\nu_k(\eta|\rho_*)} \right] \mu_{i,k}(\eta|\rho_*).
\end{align}
The logarithmic Sobolev inequality for simple exclusion \cite{Yau97} yields that there is a universal constant $C_\mathrm{LS}$, such that
\begin{align}
  H_{i,k}(\rho_*) \le C_\mathrm{LS}k^2\sum_{\eta\in\Omega_{k,\rho_*}}
\sum_{i'=1}^{2k-2} \left(\sqrt{\mu_{i,k}(\eta^{i',i'+1}|\rho_*)}
- \sqrt{\mu_{i,k}(\eta|\rho_*)}\right)^2.
\end{align}
Note that for $\eta\in\Omega_{k,\rho_*}$, $\bar\mu_{i,k}(\rho_*)\mu_{i,k}(\eta|\rho_*)=\mu_{i,k}(\eta)$, therefore
\begin{equation}
  \begin{aligned}
    &\sum_{\rho_* \in I_k} \bar\mu_{n,t}^{i,k}(\rho_*) \sum_{\eta\in\Omega_{k,\rho_*}} \sum_{i'=1}^{2k-2} \left(\sqrt{\mu_{i,k}(\eta^{i',i'+1}|\rho_*)} - \sqrt{\mu_{i,k}(\eta|\rho_*)}\right)^2\\
    = &\sum_{\eta\in\Omega_{2k-1}} \sum_{i'=1}^{2k-2} \left(\sqrt{\mu_{i,k}(\eta^{i',i'+1})} - \sqrt{\mu_{i.k}(\eta)}\right)^2\\
    \le &\sum_{\eta\in\Omega_n} \sum_{i'=i-k+1}^{i+k-2} \left( \sqrt{\mu(\eta^{i',i'+1})} - \sqrt{\mu(\eta)} \right)^2.
  \end{aligned}
\end{equation}
Therefore, from \eqref{eq:block1} and \eqref{eq:block2} we obtain that
\begin{equation}
  \begin{aligned}
    \sum_{i=k}^{n-k} \mathcal I_i \le\:&\frac{C_\mathrm{LS}k^2}a\sum_{i=k}^{n-k} \sum_{\eta\in\Omega_n} \sum_{i'=i-k+1}^{i+k-2} \left( \sqrt{\mu(\eta^{i',i'+1})} - \sqrt{\mu(\eta)} \right)^2\\
    &+ \frac1a\sum_{i=k}^{n-k} \sum_{\rho_* \in I_k} \bar\mu_{i,k}(\rho_*) \log\sum_{\eta\in\Omega_{k,\rho_*}} e^{a\mathcal F^2}\nu^k(\eta|\rho_*)\\
    \le\:&\frac{Ck^3}a\fD_n(\mu) + \frac na\sup_{\rho_* \in I_k} \bigg\{ \log\sum_{\eta\in\Omega_{k,\rho_*}} e^{a\mathcal F^2}\nu^k(\eta|\rho_*) \bigg\}.
  \end{aligned}
\end{equation}
The estimate then follows if we can find constants $c$, $C$ such that 
\begin{align}
\label{eq:exp}
  \log\sum_{\eta\in\Omega_{k,\rho_*}} e^{a\mathcal F^2}\nu^k(\eta|\rho_*) \le C, \quad \forall\,a < \frac{ck}\ell.
\end{align}

We are left with the proof of \eqref{eq:exp}.
Notice that $\nu_k(\,\cdot\,|\rho_*)$ is the conditional measure of the Bernoulli measure $\nu_\rho$ on $\Omega_{k,\rho_*}$ for $\rho\in(0,1)$.
Without loss of generality, assume that $f \in [0,1]$.
Hoeffding's lemma yields that
\begin{align}
  \log\sum_{\eta\in\Omega_{2k-1}} e^{a[f-g(\rho)]}\nu_\rho(\eta) \le \frac{a^2}8, \quad \forall\,a \in \bR.
\end{align}
By splitting the family $\{f(\tau_i\eta),i=0,\ldots,2k-\ell-1\}$ into independent groups and applying the generalized \Hol inequality,
\begin{align}
  \log\sum_{\eta\in\Omega_{2k-1}} \exp \left\{ a \left[ \frac1{2k-\ell}\sum_{i=0}^{2k-\ell-1} f(\tau_i\eta) - g(\rho) \right] \right\} \nu_\rho(\eta) \le \frac{\ell a^2}{8(2k-\ell)}.
\end{align}
Standard argument then shows that if $a \le \ell^{-1}(2k-\ell)$,
\begin{align}
  \log\sum_{\eta\in\Omega_{2k-1}} \exp \left\{ a\bigg|\frac1{2k-\ell}\sum_{i=0}^{2k-\ell-1} f(\tau_i\eta) - g(\rho)\bigg|^2 \right\} \nu_\rho(\eta) \le 3.
\end{align}
To obtain \eqref{eq:exp}, it suffices to replace $\nu_\rho$ with $\nu^k(\,\cdot\,|\rho_*)$. 
This step follows from the elementary estimate that
$\nu^k\big(\eta|_\Gamma = \tilde\eta\,\big|\,\rho_*\big) \le C\nu_{\rho_*}(\eta|_\Gamma = \tilde\eta)$ for any subset $\Gamma \subseteq \{1,\ldots,2k-1\}$ such that $|\Gamma| \le k$.

For the second term in \eqref{eq:block0}, we only need to observe that
\begin{align}
  \big|g(\bar\eta_{i+k-1,2k-1})-g\big(\hat\eta_{i,k}\big)\big| \le C\big|\bar\eta_{i+k-1,2k-1}-\hat\eta_{i,k}\big|.
\end{align}
The same upper bound holds by repeating the procedure with $f=\eta_1$.
\end{proof}

\begin{lem}
\label{lem:dir}
For any $t_0>0$, there exists a constant $C(t_0)$, such that
\begin{align}
  \int_0^{t_0} \fD_n(\mu_{n,t})dt \le C(t_0)n^{-\kappa}.
\end{align}
\end{lem}

\begin{proof}
For a probability measure $\mu$ on $\Omega_n$, its entropy is defined by $H(\mu):=\sum_\eta \mu(\eta)\log\mu(\eta)$, cf. \eqref{eq:rel-ent}.
Standard manipulation with Kolmogorov equation gives that
\begin{align}
  \frac1n\frac d{dt} H(\mu_{n,t}) = \sum_{\eta\in\Omega_n} \mu_{n,t}(\eta) \big( n^{-1}L_{n,t} \big) \big[\log\mu_{n,t}(\eta)\big].
\end{align}
Recall that $n^{-1}L_{n,t}=pL_\ta+\sigma n^\kappa L_\ss+n^\theta L_{-,t}+n^\theta L_{+,t}$.
Exploiting the inequality $x(\log y-\log x)\le2\sqrt x(\sqrt y-\sqrt x)$ for all $x$, $y>0$, we have
\begin{equation}
  \begin{aligned}
    \sum_\eta \mu_{n,t}\big(pL_\ta+\sigma n^\kappa L_\ss\big) [\mu_{n,t}] \le 2\sum_\eta \sqrt{\mu_{n,t}}\big(pL_\ta+\sigma n^\kappa L_\ss\big) \big[\sqrt{\mu_{n,t}}\big]\\
    = -\big(p+2\sigma n^\kappa\big)\fD_n(\mu_{n,t})+pE_{\mu_{n,t}} [\eta_n-\eta_1].
  \end{aligned}
\end{equation}
For the boundary operators,
\begin{equation}
  \begin{aligned}
    &\sum_\eta \mu_{n,t}\big(n^\theta L_{-,t}\big)[\mu_{n,t}] \le 2n^\theta\sum_\eta \sqrt{\mu_{n,t}} L_{-,t} \big[\sqrt{\mu_{n,t}}\big]\\
    =\,&-n^\theta\sum_\eta \big[\alpha(t)(1-\eta_1)+\gamma(t)\eta_1\big] \left( \sqrt{\mu_{n,t}(\eta^1)}-\sqrt{\mu_{n,t}(\eta)} \right)^2\\
    &+n^\theta\big[\alpha(t)-\gamma(t)\big]\sum_\eta (1-2\eta_1)\mu_{n,t}(\eta).
  \end{aligned}
\end{equation}
The same calculation is applicable on $L_{+,t}$.
Since $\alpha$, $\beta$, $\gamma$, $\delta\in L^\infty$,
\begin{align}
\label{eq:d1}
  \frac1n\frac d{dt}H(\mu_{n,t}) \le -\left( \frac p2+\sigma n^\kappa \right) \fD_n(\mu_{n,t})+C\big(1+n^\theta\big).
\end{align}
As $\theta\le0$ and $H(\mu_{n,0})=O(n)$, the proof is completed by integrating in time.
\end{proof}

\begin{rem}
\label{rem:separ1}
For the generator defined in Remark \ref{rem:separ}, the last term in \eqref{eq:d1} becomes $C(1+n^{\theta_{-,1}}+n^{\theta_{-,2}})$, which does not change the argument essentially.
\end{rem}

\section{Compensated compactness}
\label{sec:com-com}

The aim of this section is to prove that $\hat\eta_{i,k}$ converges weakly to some measure $\fQ$ in sense of \eqref{eq:hl1}.
Hereafter we choose some arbitrary $T>0$ and restrict our argument to $[0,T]\times[0,1]$ to avoid additional difficulties in compactness.

Let $\Sigma_T=[0,T]\times[0,1]$.
By a Young measure on $\Sigma_T$ taking values from $[0,1]$ we mean a family $\{\nu_{t,x};(t,x)\in\Sigma_T\}$ of probability measures on $[0,1]$, such that the mapping $(t,x)\mapsto \int f(t,x,y)\nu_{t,x}(dy)$ is measurable for all $f\in\cC(\Sigma_T\times[0,1])$.
Let $\mathcal Y$ be the space of all Young measures, endowed with the vague topology: $\nu^m\to\nu$ if and only if
\begin{align}
  \lim_{m\to\infty} \iint_{\Sigma_T} \psi(t,x) \left[ \int f\,d\nu_{t,x}^m \right] dx\,dt = \iint_{\Sigma_T} \psi(t,x) \left[ \int f\,d\nu_{t,x} \right] dx\,dt
\end{align}
for all $f\in\cC([0,1])$ and $\psi \in L^1(\Sigma_T)$.
Observe that as a topological space, $\mathcal Y$ is then metrisable, separable and compact.

Recall the smoothly weighted average $\hat\eta_{i,k}$ in \eqref{eq:hateta} with the mesoscopic scale $k=k_n$ in \eqref{eq:assp-k}.
The corresponding empirical measure process reads
\begin{align}
\label{eq:empirical}
  \nu_{t,x}^n(dy) := \sum_{i=k+1}^{n-k} \chi_{i,n}(x)\delta_{\hat\eta_{i,k}(t)}(dy), \quad (t,x)\in\Sigma_T,
\end{align}
where $\delta_u$ is the Dirac measure concentrated on $u$, and $\chi_{i,n}$ is the indicator function
\begin{align}
\label{eq:indicator}
  \chi_{i,n}(x) := \mathbf1 \left\{ x \in \left[ \frac in-\frac1{2n},\frac in+\frac1{2n} \right) \cap[0,1] \right\}.
\end{align}
Denote by $\fQ_n$ the probability measure on $\mathcal Y$ determined by $\nu^n$.
Since $\mathcal Y$ is compact, we can subtract a weakly convergent subsequence $\{\fD_{n'}\}$.
Without confusion we denote the subsequence still by $\fQ_n$.
Suppose that $\fQ$ is its weak limit.
Our main purpose is to prove that $\fQ$ is concentrated on Dirac-type Young measures.

\begin{prop}
\label{prop:hl1}
Let $\mathcal Y_D$ be the subset of $\mathcal Y$ given by
\begin{align}
  \mathcal Y_D := \Big\{\nu\in\mathcal Y; \ \exists\,\rho \in L^\infty(\Sigma_T), \ \text{s.t.} \ \nu_{t,x}=\delta_{\rho(t,x)}, \ \text{a.e. in} \ \Sigma_T\Big\}.
\end{align}
Then, $\fQ(\mathcal Y_D)=1$.
Hence, $\fQ$ determines, in the nature way, a probability measure on $L^\infty(\Sigma_T)$ which is still denoted by $\fQ$ and \eqref{eq:hl1} then follows.
\end{prop}

Before prove Proposition \ref{prop:hl1}, we give a straightforward corollary.

\begin{cor}
\label{cor:hl2}
For any $f=f(\eta_1,\ldots,\eta_\ell)$ and $\psi \in L^1(\Sigma_T)$,
\begin{align}
  \lim_{n\to\infty} \bE_n \left[ \int_0^T \frac1n\sum_{i=0}^{n-\ell} f(\tau_i\eta)\psi \left( t,\frac in \right) dt \right] = E^\fQ \left[ \iint_{\Sigma_T} \langle f \rangle(\rho)\psi\,dx\,dt \right].
\end{align}
\end{cor}

\begin{proof}
The conclusion for $\psi\in\cC(\Sigma_T)$ follows directly from Proposition \ref{prop:local-equil} and \ref{prop:hl1}.
To extend the result to $\psi \in L^1(\Sigma_T)$, we only need to apply the standard argument to approximate $\psi$ by smooth functions.
\end{proof}

Now we proceed to the proof of Proposition \ref{prop:hl1}.
We first point out the main obstacle.
Apparently, $\fQ_n(\mathcal Y_D) =1$ for each $n$.
However, the concentration is not autonomously inherited by the weak limit $\fQ$ of $\fQ_n$, since $\mathcal Y_D$ is not closed under the vague topology.
To solve this problem, the stochastic compensated compactness is introduced in \cite[Proposition 2.1]{Fritz04}, \cite[Lemma 1, Lemma 8]{FritzT04} as a sufficient condition for $\fQ(\mathcal Y_D)=1$.

Let $\cC_c^\infty(\Sigma_T)$ be the class of smooth functions on $\Sigma_T$ with compact support included in $\Sigma_T$.
Define the Sobolev norm
\begin{align}
  \|\psi\|_{H_0^1}^2 := \|\psi\|_{L^2}^2 + \|\partial_t\psi\|_{L^2}^2 + \|\partial_x\psi\|_{L^2}^2, \quad \forall\,\psi\in\cC_c(\Sigma_T),
\end{align}
where $\|\cdot\|_{L^2}$ is the usual $L^2$ norm.
Also let $\|\cdot\|_{L^\infty}$ be the usual $L^\infty$ norm.

Let $(f,q)$ be a Lax entropy flux pair in Definition \ref{defn:ent-flux}.
Define the microscopic entropy production associated to $(f,q)$ by
\begin{align}
  X_n^{f,q}(\psi) := -\iint_{\Sigma_T} \left( \partial_t\psi\int f\,d\nu_{t,x}^n+p\partial_x\psi\int q\,d\nu_{t,x}^n \right) dx\,dt,
\end{align}
for all $\psi\in\cC^1(\Sigma_T)$.
To show Proposition \ref{prop:hl1}, it suffices to prove the next lemma.

\begin{lem}[Stochastic compensated compactness]
\label{lem:com-com}
The entropy production decomposes as $X_n^{f,q} = Y_n+Z_n$, such that
\begin{align}
  \label{eq:com-com-y}
  &|Y_n(\psi)| \le a_n\|\psi\|_{H_0^1}, \ \forall\,\psi\in\cC_c(\Sigma_T) \quad \text{and} \quad \lim_{n\to\infty} \bE_n [a_n]=0;\\
  \label{eq:com-com-z}
  &|Z_n(\psi)| \le b_n\|\psi\|_{L^\infty}, \ \forall\,\psi\in\cC_c(\Sigma_T) \quad \text{and} \quad \sup_{n\ge1} \bE_n [b_n]<\infty.
\end{align}
\end{lem}

The proof is similar to \cite[Section 6]{Xu21}, and there is no difference between the two types of reservoirs.
Recall the indicator function $\chi_{i,n}$ in \eqref{eq:indicator}.
For $\psi\in\cC(\Sigma_T)$, define
\begin{align}
  \psi_i(t) := \psi \left( \frac in - \frac1{2n}, t \right), \quad \bar\psi_i(t) := n\int_0^1 \psi(t,x) \chi_{i,n}(x)dx. 
\end{align}
For a sequence $\{a_i\}$, denote $\nabla a_i=a_{i+1}-a_i$, $\nabla^*a_i=a_{i-1}-a_i$ and
\begin{align}
  \Delta a_i=(-\nabla^*\nabla)a_i=a_{i+1}-2a_i+a_{i-1}.
\end{align}

Observe that $\hat\eta_{i,k}$ is supported on $\{\eta_2, ..., \eta_{n-1}\}$ for $i=k+1$, ... $n-k$, so $L_{\pm,t}$ does not contribute to its time evolution:
\begin{align}
\label{eq:current}
  L_{n,t} [\hat\eta_{i,k}] = n\big(p\nabla^*\hat J_{i,k}+\sigma n^\kappa\Delta\hat\eta_{i,k}\big), \quad \hat J_{i,k} := \sum_{i'=-k+1}^{k-1} w_{i'}J_{i-i',i-i'+1},
\end{align}
where $J_{i,i+1}:=\eta_i(1-\eta_{i+1})$.
In most of the following contents, we omit the subscript $k$ in $\hat\eta_{i,k}$, $\hat J_{i,k}$ and write $\hat\eta_i$, $\hat J_i$ for short.

\begin{proof}[Proof of Lemma \ref{lem:com-com}]
By the definition of $\nu^n$ in \eqref{eq:empirical},
\begin{equation}
\label{eq:ent-prod-decom0}
  \begin{aligned}
    X_n^{f,q}(\psi) = &-\int_0^T \frac1n\sum_{i=k+1}^{n-k} f(\hat\eta_i)\bar\psi'_i\,dt - p\int_0^T \sum_{i=k+1}^{n-k} q(\hat\eta_i)\nabla\psi_i\,dt\\
    &- pq(0)\int_0^T \left( \int_0^{\frac{2k+1}{2n}} + \int_{1-\frac{2k-1}{2n}}^1 \right) \partial_x\psi\,dx\,dt.
  \end{aligned}
\end{equation}
Denote by $M_i^f=M_i^f(t)$ the Dynkin's martingale associated to $f(\hat\eta_i)$:
\begin{align}
\label{eq:martingale}
  M_i^f(t) := f\big(\hat\eta_i(t)\big) - f\big(\hat\eta_i(0)\big) - \int_0^t L_{n,s} \big[f(\hat\eta_i(s))\big]ds.
\end{align}
Recall \eqref{eq:current} and define for $i=k+1$, ..., $n-k$ that
\begin{align}
\label{eq:corrections}
  \ep_i := \frac{L_{n,t}\big[f(\hat\eta_i)\big] - f'(\hat\eta_i)L_{n,t}[\hat\eta_i]}n, \quad
  \ep_i^* := f'(\hat\eta_i)\nabla^*J(\hat\eta_i) - \nabla^*q(\hat\eta_i). 
\end{align}
Using the relation $f'J'=q'$, we can decompose $L_{n,s}[f(\hat\eta_i)]$ into
\begin{align}
  n \left\{ pf'(\hat\eta_i)\nabla^* \left[ \hat J_i-J(\hat\eta_i) \right] + \sigma n^\kappa f'(\hat\eta_i)\Delta\hat\eta_i + \ep_i + p\ep_i^* + p\nabla^*q(\hat\eta_i) \right\}.
\end{align}
Choose $\psi\in\cC^1(\Sigma_T)$ such that $\psi(T,\cdot)=0$.
Performing integration by parts in the first integral in \eqref{eq:ent-prod-decom0} and using \eqref{eq:martingale}, \eqref{eq:corrections}, we obtain that
\begin{equation}
\label{eq:ent-prod-decom}
  \begin{aligned}
    X_n^{f,q}(\psi) =\;&\frac1n\sum_{i=k+1}^{n-k} f\big(\hat\eta_i(0)\big)\bar\psi_i(0)\\
    &+ \mathcal M_n(\psi) + \cA_n(\psi) + \cS_n(\psi) + \sum_{\ell=1,2,3} \mathcal E_{n,\ell}(\psi),
  \end{aligned}
\end{equation}
where $\mathcal M_n(\psi)$ is the martingale given by
\begin{align}
\label{eq:def-m}
  \mathcal M_n(\psi) := -\int_0^T \frac1n\sum_{i=k+1}^{n-k} \bar\psi'_i(t)M_i^f(t)dt,
\end{align}
with $M_i^f$ in \eqref{eq:martingale}, $\cA_n$ and $\cS_n$ are given by
\begin{align}
\label{eq:def-a}
  \cA_n(\psi) &:= p\int_0^T \sum_{i=k+1}^{n-k} \bar\psi_i(t)f'\big(\hat\eta_i(t)\big)\nabla^* \left[ \hat J_i(t) - J\big(\hat\eta_i(t)\big) \right] dt,\\
\label{eq:def-s}
  \cS_n(\psi) &:= \sigma n^\kappa\int_0^T \sum_{i=k+1}^{n-k} \bar\psi_i(t)f'\big(\hat\eta_i(t)\big)\Delta\hat\eta_i(t)dt,
\end{align}
and for $\ell=1$, $2$, $3$, $\mathcal E_{n,\ell}$ are given by
\begin{align}
\label{eq:def-e1}
  \mathcal E_{n,1}(\psi) &:= p\int_0^T \sum_{i=k+1}^{n-k} \big[\bar\psi_i\nabla^*q(\hat\eta_i) - q(\hat\eta_i)\nabla\psi_i\big]dt,\\
\label{eq:def-e2} 
  \mathcal E_{n,2}(\psi) &:= \int_0^T \sum_{i=k+1}^{n-k} (\ep_i+p\ep_i^*)\bar\psi_i\,dt,\\
\label{eq:def-e3}
  \mathcal E_{n,3}(\psi) &:= -pq(0)\int_0^T \left( \int_0^{\frac{2k+1}{2n}} + \int_{1-\frac{2k-1}{2n}}^1 \right) \partial_x\psi\,dx\,dt.
\end{align}

To close the proof,  define for each $\psi\in\cC_c^\infty(\Sigma_T)$ that
\begin{equation}
  \begin{aligned}
    &Y_n(\psi) := X_n^{f,q}(\psi)-Z_n(\psi),\\
    &Z_n(\psi) := \int_0^T \sum_{i=k+1}^{n-k} \bar\psi_i \left[ p\hat J_i-pJ\big(\hat\eta_i\big) -\sigma n^\kappa\nabla\hat\eta_i\right] \nabla f'\big(\hat\eta_i\big)dt + \mathcal E_{n,2}(\psi).
  \end{aligned}
\end{equation}
From Lemma \ref{lem:esti-y} and \ref{lem:esti-z} below, this is the desired decomposition.
\end{proof}

\begin{lem}
\label{lem:esti-y}
$Y_n=Y_n(\psi)$ satisfies the condition \eqref{eq:com-com-y}.
\end{lem}

\begin{lem}
\label{lem:esti-z}
$Z_n=Z_n(\psi)$ satisfies the condition \eqref{eq:com-com-z}.
\end{lem}

To prove Lemma \ref{lem:esti-y} and \ref{lem:esti-z}, we make use of the following block estimates.
Their proofs are parallel to Proposition \ref{prop:local-equil} and are postponed to the end of this section.

\begin{prop}[One-block estimate]
\label{prop:one-block}
There is some constant $C$, such that 
\begin{align}
  \bE_n \left[ \int_0^T \sum_{i=k+1}^{n-k} \left[ \hat J_i-J\big(\hat\eta_i\big) \right]^2dt \right] \le C \left( \frac{k^2}{\sigma n^\kappa} + \frac nk \right). 
\end{align}
\end{prop}

\begin{prop}[$H^1$ estimate]
\label{prop:h1}
There is some constant $C$, such that 
\begin{align}
  \bE_n \left[ \int_0^T \sum_{i=k+1}^{n-k} \big(\nabla\hat\eta_i\big)^2dt \right] \le C \left( \frac1{\sigma n^\kappa} + \frac n{k^3} \right). 
\end{align}
\end{prop}

\begin{proof}[Proof of Lemma \ref{lem:esti-y}]
For $\psi\in\cC_c^\infty(\Sigma_T)$, define
\begin{align}
\label{eq:def-a1}
  \cA_{n,1}(\psi) &:= p\int_0^T \sum_{i=k+1}^{n-k} \bar\psi_i\left[\hat J_i-J\big(\hat\eta_i\big) \right]\nabla f'(\hat\eta_i)dt,\\
\label{eq:def-s1}
  \cS_{n,1}(\psi) &:= -\sigma n^\kappa\int_0^T \sum_{i=k+1}^{n-k} \bar\psi_i\nabla\hat\eta_i\nabla f'(\hat\eta_i)dt.
\end{align}
In view of \eqref{eq:ent-prod-decom}, $Y_n=\mathcal M_n + (\cA_n-\cA_{n,1}) + (\cS_n-\cS_{n,1}) + \mathcal E_{n,1} + \mathcal E_{n,3}$.

We first look at $\cA_n-\cA_{n,1}$.
Write $g_i:=\hat J_i-J(\hat\eta_i)$ and fix $p=1$ without loss of generality.
Performing summation by part, $\cA_n(\psi)-\cA_{n,1}(\psi)$ is equal to\footnote{Though $\psi$ is compactly supported, the boundary terms cannot be omitted autonomously, since we need to take supreme over all $\psi$ before send $n\to\infty$.}
\begin{align}
\label{eq:y1}
  \int_0^T \sum_{i=k+1}^{n-k} f'\big(\hat\eta_{i+1}\big)g_i\nabla\bar\psi_i\,dt - \int_0^T \left( \bar\psi_{i+1}f'\big(\hat\eta_{i+1}\big)g_i\Big|_{i=k}^{n-k} \right) dt.
\end{align}
Notice that for compactly supported $\psi$,
\begin{align}
\label{eq:esti-bd}
  \left| \int_0^T \bar\psi_{k+1}\,dt \right| = \left| \int_0^T \!\! \int_0^{\frac kn+\frac1{2n}} \partial_x\psi\,dx\,dt \right| \le C\sqrt{\frac kn}\|\psi\|_{H_0^1}.
\end{align}
Similar estimate holds for $\bar\psi_{n-k+1}$.
Hence, the second term in \eqref{eq:y1} satisfies the condition.
For the first term, \CS inequality yields that
\begin{align}
  \left| \int_0^T \sum_{i=k+1}^{n-k} f'\big(\hat\eta_{i+1}\big)g_i\nabla\bar\psi_i\,dt \right| \le C|f'|_\infty\bigg(\frac1n\int_0^T \sum_{i=k+1}^{n-k} g_i^2dt\bigg)^{\frac12}\|\psi\|_{H_0^1}.
\end{align}
Applying Proposition \ref{prop:one-block}, it is bounded from above by $a_{n,1}\|\psi\|_{H_0^1}$ and
\begin{align}
  \bE_n [a_{n,1}] \le \frac{C'}{\sqrt n}\sqrt{\frac{k^2}{n^\kappa}+\frac nk} = C'\sqrt{\frac{k^2}{n^{1+\kappa}}+\frac1k}.
\end{align}

The term $\cS_n-\cS_{n,1}$ is treated similarly.
Without loss of generality, we fix $\sigma=1$.
Write $\cS_n(\psi)-\cS_{n,1}(\psi)$ as
\begin{align}
  n^\kappa\int_0^T \left( \bar\psi_{i+1}f'\big(\hat\eta_{i+1}\big)\nabla\hat\eta_i\Big|_{i=k}^{n-k} \right) dt - n^\kappa\int_0^T \sum_{i=k+1}^{n-k} f'\big(\hat\eta_{i+1}\big)\nabla\hat\eta_i\nabla\bar\psi_i\,dt.
\end{align}
Since $|\nabla\hat\eta_i| \le k^{-1}$ and \eqref{eq:esti-bd}, the first term is bounded by $Cn^{\kappa-\frac12}k^{-\frac12}\|\psi\|_{H_0^1}$.
By \CS inequality and Proposition \ref{prop:h1}, the second term is bounded by $a_{n,2}\|\psi\|_{H_0^1}$ and
\begin{align}
  \bE_n [a_{n,2}] \le \frac{Cn^\kappa}{\sqrt n}\sqrt{\frac1{n^\kappa}+\frac n{k^3}} = C\sqrt{\frac1{n^{1-\kappa}}+\frac{n^{2\kappa}}{k^3}}.
\end{align}

Next, apply summation by part on $\mathcal E_{n,1}$ in \eqref{eq:def-e1} to get
\begin{equation}
  \begin{aligned}
    \mathcal E_{n,1}(\psi) = \int_0^T \sum_{i=k+1}^{n-k} (\bar\psi_i-\psi_i)\nabla^*q\big(\hat\eta_i\big)dt - \int_0^T \left( \psi_{i+1}q\big(\hat\eta_i\big)\Big|_{i=k}^{n-k} \right) dt.
  \end{aligned}
\end{equation}
The boundary integral is estimated by \eqref{eq:esti-bd} as before.
For the remaining term, notice that $|\nabla^*q(\hat\eta_i)| \le |q'|_\infty|\nabla^*\hat\eta_i|$ and $|\nabla^*\hat\eta_i| \le k^{-1}$.
\CS inequality then yields that this term is bounded from above by
\begin{equation}
  \begin{aligned}
    \frac{p|q'|_\infty T\sqrt n}k \bigg[\int_0^T \sum_{i=k+1}^{n-k} (\bar\psi_i-\psi_i)^2dt\bigg]^{\frac12} \le \frac Ck\|\psi\|_{H_0^1}.
  \end{aligned}
\end{equation}
For $\mathcal E_{n,3}$ in \eqref{eq:def-e3}, we similarly have
\begin{align}
  \big|\mathcal E_{n,3}(\psi)\big| \le \frac{2p|q(0)|\sqrt{Tk}}{\sqrt n} \left[ \iint_{\Sigma_T} (\partial_x\psi)^2dx\,dt \right]^\frac12 \le C\sqrt{\frac kn}\|\psi\|_{H_0^1}.
\end{align}

Finally, we check the martingale $\mathcal M_n$.
By \eqref{eq:def-m},
\begin{align}
  \big|\mathcal M_n(\psi)\big|^2 &\le C\|\psi\|_{H_0^1}^2\int_0^T \frac1n\sum_{i=k+1}^{n-k} \big|M_i^f(t)\big|^2dt.
\end{align}
Doob's inequality then yields that $|\mathcal M_n(\psi)| \le a_{n,3}\|\psi\|_{H_{0,1}}$, where
\begin{align}
  \bE_n \big[|a_{n,3}|^2\big] &\le \frac{C'}n\sum_{i=k+1}^{n-k} \bE_n \left[ \int_0^T \big\langle M_i^f \big\rangle(t)dt \right],
\end{align}
where the quadratic variation $\langle M_i^f \rangle$ reads
\begin{equation}
\label{eq:martingale1}
  \begin{aligned}
    \big\langle M_i^f \big\rangle(t) &= \int_0^t \Big\{L_{n,s} \big[f(\hat\eta_i)^2\big]-2f(\hat\eta_i)L_{n,s} \big[f(\hat\eta_i)\big]\Big\}ds\\
    &=n\int_0^t \sum_{i'=1}^{n-1} (p\eta_{i'}+\sigma n^\kappa) \left[ f\big(\hat\eta_i^{i',i'+1}\big) - f\big(\hat\eta_i\big) \right]^2ds.
  \end{aligned}
\end{equation}
Recall that $\hat\eta_i=\hat\eta_{i,k}$ in \eqref{eq:hateta}.
Direct computation shows that
\begin{equation}
\label{eq:exchange-smooth}
  \hat\eta_i^{j,j+1} = \hat\eta_i - \sgn\left( i-j-\frac12 \right)\frac{\nabla\eta_j}{k^2}, \quad j-k+1 \le i \le j+k
\end{equation}
and otherwise $\hat\eta_i^{j,j+1}-\hat\eta_i=0$.
Therefore,
\begin{align}
  \bE_n \big[a_{n,3}^2\big] \le Cn^{1+\kappa}\int_0^T \!\! \int_0^t \sum_{i'=1}^{n-1} \left( \hat\eta_i^{i',i'+1} - \hat\eta_i \right)^2ds\,dt \le \frac{C'n^{1+\kappa}}{k^3}.
\end{align}
In view of \eqref{eq:assp-k}, we choose $k$ such that all the upper bounds above vanish when $n\to\infty$.
The proof is then completed.
\end{proof}

\begin{proof}[Proof of Lemma \ref{lem:esti-z}]
Recall that $Z_n = \cA_{n,1} + \cS_{n,1} + \mathcal E_{n,2}$.
Similarly to Lemma \ref{lem:esti-y}, we prove for each term in $Z_n$.

Recall the definition \eqref{eq:def-a1} of $\cA_{n,1}$ and write $g_i=\hat J_i-J(\hat\eta_i)$ as before.
Since $|\nabla f'(\hat\eta_i)|\le|f''|_\infty|\nabla\eta_i|$, we have $|\cA_{n,1}(\psi)| \le b_{n,1}\|\psi\|_{L^\infty}$, where
\begin{align}
  b_{n,1} := C\int_0^T \sum_{i=k+1}^{n-k} \big|g_i\nabla\hat\eta_i\big|dt.
\end{align}
Splitting $\bE_n [b_{n,1}]$ into the product of two expectations by \CS inequality and applying Proposition \ref{prop:one-block}, \ref{prop:h1} respectively,
\begin{align}
    \bE_n [b_{n,1}] \le C'\sqrt{\frac{k^2}{n^\kappa}+\frac nk}\sqrt{\frac1{n^\kappa}+\frac n{k^3}} \le C'' \left( \frac k{n^\kappa}+\frac n{k^2} \right).
\end{align}

For $\cS_{n,1}$ in \eqref{eq:def-s1}, similarly we have $|\cS_{n,1}(\psi)| \le b_{n,2}\|\psi\|_{L^\infty}$, where
\begin{align}
  b_{n,2} := Cn^\kappa\int_0^T \sum_{i=k+1}^{n-k} \big(\nabla\hat\eta_i\big)^2dt.
\end{align}
According to Proposition \ref{prop:h1},
\begin{align}
  \bE_n [b_{n,2}] \le C'n^\kappa \left( \frac1{n^\kappa} + \frac n{k^3} \right) = C'' \left( 1+\frac{n^{1+\kappa}}{k^3} \right).
\end{align}

We are left with $\mathcal E_{n,2}$ in \eqref{eq:def-e2}.
From \eqref{eq:corrections} and the definition \eqref{eq:generator} of $L_{n,t}$,
\begin{align}
  \ep_i = \sum_{j =i-k}^{i+k-1} \big(p\eta_j+\sigma n^\kappa\big) \left[ f(\hat\eta_i^{j,j+1}) - f(\hat\eta_i) - f'(\hat\eta_i) \left( \hat\eta_i^{j,j+1} - \hat\eta_i \right) \right].
\end{align}
Thanks to \eqref{eq:exchange-smooth} and the mean value theorem, $|\ep_i| \le Cn^\kappa k^{-3}$.
For $\ep_i^*$, it follows from the relation $f'J'=q'$ that $|\ep_i^*| \le Ck^{-2}$.
Hence,
\begin{align}
  |\mathcal E_{n,2}(\psi)| \le C \left( \frac{n^{1+\kappa}}{k^3}+\frac n{k^2} \right) \|\psi\|_{L^\infty}.
\end{align}
The conclusion then follows from our choice of $k$ in \eqref{eq:assp-k}.
\end{proof}

\begin{rem}
\label{rem:ent-prod}
It is clear that $\cS_{n,1}$ is the only term that survives in the limit $n\to\infty$.
It is the microscopic origin of the non-zero entropy production appeared in \eqref{eq:ent-sol-1}.
\end{rem}

We close this section with the proofs of the block estimates.

\begin{proof}[Proof of Proposition \ref{prop:one-block}]
Recall that $\hat J_i=\hat J_{i,k}$ is defined in \eqref{eq:current}.
Take $f(\eta_1,\eta_2)=\eta_1(1-\eta_2)$, then $J_{i,i+1}=f(\tau_{i-1}\eta)$.
Define weighted average
\begin{align}
  \hat f_{i,k} := \sum_{i'=-k}^{k-2} \frac{k-|i'+1|}{k^2}f(\tau_{i+i'}\eta) = \hat J_i.
\end{align}
We only need to repeat the proof of Lemma \ref{lem:block} with the uniform average $f_{i,k}$ replaced by $\hat f_{i,k}$ to get for any probability measure $\mu$ on $\Omega_n$ that,
\begin{align}
  \sum_{\eta\in\Omega_n} \sum_{i=k}^{n-k} \big|\hat J_i-J\big(\hat\eta_i\big)\big|^2\mu(\eta) \le C \left[ k^2\fD_n(\mu)+\frac nk \right].
\end{align}
Taking $\mu=\mu_{n,t}$ and integrating in time, we can conclude from Lemma \ref{lem:dir}.
\end{proof}

\begin{proof}[Proof of Proposition \ref{prop:h1}]
Observe that $\nabla\hat\eta_i = k^{-1}(\bar\eta_{i+k}-\bar\eta_i)$.
Applying the same argument used in the previous proof, we have
\begin{align}
  \bE_n \left[ \int_0^T \sum_{i=k+1}^{n-k} \big(\bar\eta_{i+k}-\bar\eta_i\big)^2dt \right] \le C \left( \frac{k^2}{n^\kappa} + \frac nk \right).
\end{align}
The estimate then follows from dividing the formula above by $k^2$.
\end{proof}

\section{Entropy inequality and boundary traces}
\label{sec:ent-sol}

In Proposition \ref{prop:hl1}, for each $T>0$ we obtain a random function $\rho \in L^\infty(\Sigma_T)$ with distribution $\fQ$ such that Corollary \ref{cor:hl2} holds.
In this section, we check conditions (\romannum1)--(\romannum3) in Definition \ref{defn:ent-sol} for $\fQ$-almost every $\rho$.

We begin with the entropy inequality \eqref{eq:ent-sol-1}.
It is a direct conclusion of the computation in the previous section.
Indeed, as mentioned in Remark \ref{rem:ent-prod}, all terms in the entropy production vanish except $\cS_{n,1}$ in \eqref{eq:def-s1}.
Recall that $f$ is convex, then $(u_1-u_2)[f'(u_1)-f'(u_2)]\ge0$ for all $u_1$, $u_2\in[0,1]$.
Therefore, $\cS_{n,1}(\psi)\le0$ if $\psi\ge0$, which implies that
\begin{align}
  \lim_{n\to\infty} \fQ_n \left\{ X_n^{f,q}(\psi) \le 0 \right\}=1, \quad \forall (f,q)\in\mathscr S, \ \psi\in\cC_c^\infty(\Sigma_T), \ \psi\ge0.
\end{align}
The weak convergence $\fQ_n\Rightarrow\fQ$ then shows that
\begin{align}
  \fQ \left\{ \iint_{\Sigma_T} \big[f(\rho)\partial_t\psi+pq(\rho)\partial_x\psi\big]dx\,dt \ge 0 \right\} = 1,
\end{align}
for any fixed $\psi\ge0$ and $(f,q)$.
From Remark \ref{rem:ent-sol} and that $\cC_c^\infty(\Sigma_T)$ is separable, we have the next lemma.

\begin{lem}
\label{lem:ent-ineq}
Let $\mathscr S$ be the set of all Lax entropy flux pair, then
\begin{equation}
  \fQ \left\{ \iint_{\Sigma_T} \big[f(\rho)\partial_t\psi+pq(\rho)\partial_x\psi\big]dx\,dt \ge 0
  \ \bigg| \ 
  \begin{aligned}
    &\forall\,(f,q)\in\mathscr S\\
    &\forall\,\psi\in\cC_c^\infty(\Sigma_T)
  \end{aligned}
  \right\} = 1.
\end{equation}
\end{lem}

By Lemma \ref{lem:ent-ineq} and Proposition \ref{prop:bd-trace}, $\rho$ possesses traces at $t=0$, $x=0$ and $x=1$ as in \eqref{eq:bd-traces1}, $\fQ$-almost surely.
This allows us to transfer the initial condition in \eqref{eq:ent-sol-2} and boundary conditions in \eqref{eq:ent-sol-3} into the corresponding weak forms.

\begin{lem}
\label{lem:weak-sol}
For any $\psi\in\cC_c^\infty((-\infty,T)\times(0,1))$, we have $\fQ$-almost surely that
\begin{align}
  \int_0^1 v_0\psi(0,\cdot)dx + \iint_{\Sigma_T} \big[\rho\partial_t\psi + pJ(\rho)\partial_x\psi\big]dx\,dt = 0.
\end{align}
\end{lem}

\begin{lem}
\label{lem:weak-bd}
For any $\phi\in\cC^1([0,T])$, $\phi\ge0$, we have $\fQ$-almost surely that
\begin{equation}
  \begin{aligned}
    &\lim_{\ve\to0+} \frac1\ve\int_0^T \phi(t) \left[ \int_0^\ve J\big(\rho(t,x)\big)dx \right] dt = 0,\\
    &\lim_{\ve\to0+} \frac1\ve\int_0^T \phi(t) \left[ \int_0^{1-\ve} J\big(\rho(t,x)\big)dx \right] dt = 0.
  \end{aligned}
\end{equation}
\end{lem}

The proof of Lemma \ref{lem:weak-sol} is straightforward.
By taking $(f,q)=(\mathrm{id},pJ)$ in \eqref{eq:ent-prod-decom}, Lemma \ref{lem:weak-sol} follows from  \eqref{eq:assp-initial} and the fact that $\cS_{n,1}\equiv0$.
Hereafter we focus on Lemma \ref{lem:weak-bd}.

\begin{proof}
Let $j_{i,i+1}$ be the microscopic current determined by the conservation law $L_{n,t} [\eta_i] = j_{i-1,i}-j_{i,i+1}$.
It is easy to see that (cf. \eqref{eq:current})
\begin{align}
\label{eq:current1}
  j_{i,i+1}=
  \begin{cases}
    n^{1+\theta}[\alpha(t)-(\alpha(t)+\gamma(t))\eta_1], &i=0,\\
    pnj_\ta(\tau_{i-1}\eta)-\sigma n^{1+\kappa}\nabla\eta_i, &1 \le i \le n-1,\\
    n^{1+\theta}[(\beta(t)+\delta(t))\eta_n-\delta(t)], &i=n,
  \end{cases}
\end{align}
where $\nabla\eta_i=\eta_i-\eta_{i+1}$ and $j_\ta(\eta):=\eta_1(1-\eta_2)$.

For $\phi\in\cC^1([0,T])$, $\phi\ge0$, consider the Dynkin's martingale
\begin{equation}
  \begin{aligned}
    M_i^\phi(t) :=\,&\sum_{i'=1}^i \left[ \phi(t)\eta_{i'}(t)-\phi(0)\eta_{i'}(0) - \int_0^t \phi'(s)\eta_{i'}(s)ds \right]\\
    &-\int_0^t \phi(s)\sum_{i'=1}^i L_{n,t} [\eta_{i'}(s)]ds, \quad \forall\,t\in[0,T].
  \end{aligned}
\end{equation}
Since $L_{n,t} [\eta_{i'}]=j_{i'-1,i'}-j_{i',i'+1}$ with $j_{i',i'+1}$ in \eqref{eq:current1},
\begin{equation}
  \begin{aligned}
    &\int_0^t \phi(s)j_\ta\big(\tau_{i-1}\eta(s)\big)ds + \frac{\sigma n^\kappa}p\int_0^t \phi(s)\nabla\eta_i(s)ds\\
    =\;&\frac{n^\theta}p\int_0^t \phi(s)\big[\alpha(s)-(\alpha(s)+\gamma(s))\eta_1(s)\big]ds + \frac1{np}M_i^\phi(t) \\
    &- \frac1{np}\sum_{i'=1}^i \left[ \phi(t)\eta_{i'}(t) - \phi(0)\eta_{i'}(0) - \int_0^t \phi'(s)\eta_{i'}(s)ds \right].
  \end{aligned}
\end{equation}
Fix any $\ve\in(0,1)$ and sum up the equation for $i=1$, ..., $[\ve n]$ to get
\begin{equation}
  \begin{aligned}
    &\int_0^t \sum_{i=1}^{[\ve n]} \phi(s)j_\ta(\tau_{i-1}\eta)ds + \frac{\sigma n^\kappa}p\int_0^t \phi(s)(\eta_{[\ve n]+1}-\eta_1)ds\\
    =\;&\frac{n^\theta[\ve n]}p \int_0^t \phi(s)\big[\alpha(s)-(\alpha(s)+\gamma(s))\eta_1(s)\big]ds + \frac1{np}\sum_{i=1}^{[\ve n]} M_i^\phi(t)\\
    &-\frac1{np}\sum_{i=1}^{[\ve n]} \sum_{i'=1}^i \left[ \phi(t)\eta_{i'}(t) - \phi(0)\eta_{i'}(0) - \int_0^t \phi'(s)\eta_{i'}(s)ds \right].
  \end{aligned}
\end{equation}
Take expectation with respect to $\bP_n$ in the formula above.
Since $M_i^\phi(t)$ is mean-zero and the integral of $\phi(s)(\alpha(s)+\gamma(s))\eta_1(s)$ is non-negative,
\begin{align}
  \bE_n \left[ \int_0^T \sum_{i=1}^{[\ve n]} \phi(t)j_\ta(\tau_{i-1}\eta)dt \right] \le C\big(n^\kappa+\ve n^{1+\theta}\|\alpha\|_{L^\infty}+\ve^2n),
\end{align}
where $C$ is a constant depending on $|\phi|_\infty$, $|\phi'|_\infty$ and $T$.
Observe that in both cases of Theorem \ref{thm:slow}, the second term in the right-hand side above is $o(\ve n)$.
Apply Corollary \ref{cor:hl2} with $j_\ta=\eta_1(1-\eta_2)$ and the $L^1$ function $\psi_\ve(t,x)=\ve^{-1}\phi(t)\mathbf1_{[0,\ve]}(x)$,
\begin{equation}
  \begin{aligned}
    &E^\fQ \left[ \frac1\ve\int_0^T \!\! \int_0^\ve \phi(t)J\big(\rho(t,x)\big)dx\,dt \right]\\
    =\,&\lim_{n\to\infty} \bE_n \left[ \int_0^t \frac1{\ve n}\sum_{i=1}^{[\ve n]} \phi(t)j_\ta(\tau_{i-1}\eta)dt \right] \le C\ve,
  \end{aligned}
\end{equation}
where in the last inequality we need $\kappa<1$.
Hence, for any $\phi\in\cC^1([0,T])$,
\begin{align}
  \lim_{\ve\to0+} E^\fQ \left[ \frac1\ve\int_0^T \phi(t)\int_0^\ve J\big(\rho(t,x)\big)dx\,dt \right] = 0.
\end{align}
The other assertion follows similarly.
\end{proof}

\begin{rem}
\label{rem:separ2}
For the generator defined in Remark \ref{rem:separ}, the microscopic currents at the boundaries become
\begin{equation}
  \begin{aligned}
    j'_{0,1}&=n^{1+\theta_{-,1}}\alpha(t)(1-\eta_1)-n^{1+\theta_{-,2}}\gamma(t)\eta_1,\\
    j'_{n,n+1}&=n^{1+\theta_{+,1}}\beta(t)\eta_n-n^{1+\theta_{+,2}}\delta(t)(1-\eta_n).
  \end{aligned}
\end{equation}
Repeating the manipulations, we obtain that
\begin{align}
  \bE_n \left[ \int_0^T \sum_{i=1}^{[\ve n]} \phi(t)j_\ta(\tau_{i-1}\eta)dt \right] \le C\big(n^\kappa+\ve n^{1+\theta_{-,1}}\|\alpha\|_{L^\infty}+\ve^2n).
\end{align}
Since $\theta_{\pm,1}<0$, the remaining argument still works.
\end{rem}


\titleformat{\section}[hang]	
{\bfseries\large}{}{0em}{}[]
\titlespacing*{\section}{0em}{2em}{1.5em}

\section{References}	
\renewcommand{\section}[2]{}
\bibliography{bibliography.bib}



\vspace{2em}
\noindent{\large Lu \textsc{Xu}}

\vspace{0.5em}
\noindent Centre de recherche Inria Lille - Nord Europe\\
40 Avenue Halley, 59650 Villeneuve-d'Ascq, France\\
{\tt lu.xu@inria.fr}

\end{document}